\documentclass[english,3p]{elsarticle}
\usepackage[T1]{fontenc}
\usepackage[utf8]{inputenc}
\usepackage{float}
\usepackage{bm}
\usepackage{amsmath}
\usepackage{amsthm}
\usepackage{amssymb}
\usepackage{stmaryrd}
\usepackage{graphicx}
\usepackage{esint}
\usepackage{subscript}

\makeatletter
\numberwithin{equation}{section}
\numberwithin{figure}{section}
\theoremstyle{plain}
\newtheorem{thm}{\protect\theoremname}
\theoremstyle{plain}
\newtheorem{lem}[thm]{\protect\lemmaname}
\theoremstyle{remark}
\newtheorem{rem}[thm]{\protect\remarkname}
\theoremstyle{plain}
\newtheorem{prop}[thm]{\protect\propositionname}

\usepackage{algpseudocode,algorithm,algorithmicx}
\usepackage{tikz}
\usetikzlibrary{shapes,arrows,calc,decorations.pathreplacing}
\usepackage{lineno}
\@ifundefined{showcaptionsetup}{}{%
 \PassOptionsToPackage{caption=false}{subfig}}
\usepackage{subfig}
\makeatother

\usepackage{babel}
\providecommand{\lemmaname}{Lemma}
\providecommand{\propositionname}{Proposition}
\providecommand{\remarkname}{Remark}
\providecommand{\theoremname}{Theorem}

\begin{document}

\title{Conservative, pressure-equilibrium-preserving discontinuous Galerkin
method for compressible, multicomponent flows}

\author[NRL]{Eric J. Ching}
\author[NRL]{Ryan F. Johnson}
\author[NRL]{Andrew D. Kercher}

\address[NRL]{Laboratories for Computational Physics and Fluid Dynamics, U.S. Naval Research Laboratory, 4555 Overlook Ave SW, Washington, DC 20375}

\begin{abstract}
This paper concerns preservation of velocity and pressure equilibria
in smooth, compressible, multicomponent flows in the inviscid limit.
First, we derive the velocity-equilibrium and pressure-equilibrium
conditions of a standard discontinuous Galerkin method that discretizes
the conservative form of the compressible, multicomponent Euler equations.
We show that under certain constraints on the numerical flux, the
scheme is velocity-equilibrium-preserving. However, standard discontinuous
Galerkin schemes are not pressure-equilibrium-preserving. Therefore,
we introduce a discontinuous Galerkin method that discretizes the
pressure-evolution equation in place of the total-energy conservation
equation. Semidiscrete conservation of total energy, which would otherwise
be lost, is restored via the correction terms of Abgrall~\citep{Abg18}
and Abgrall et al.~\citep{Abg22}. Since the addition of the correction
terms prevents exact preservation of pressure and velocity equilibria,
we propose modifications that then lead to a velocity-equilibrium-preserving,
pressure-equilibrium-preserving, and (semidiscretely) energy-conservative
discontinuous Galerkin scheme, although there are certain tradeoffs.
Additional extensions are also introduced. We apply the developed
scheme to smooth, interfacial flows involving mixtures of thermally
perfect gases initially in pressure and velocity equilibria to demonstrate
its performance in one, two, and three spatial dimensions. 
\end{abstract}
\maketitle
\global\long\def\middlebar{\,\middle|\,}%
\global\long\def\average#1{\left\{  \!\!\left\{  #1\right\}  \!\!\right\}  }%
\global\long\def\expnumber#1#2{{#1}\mathrm{e}{#2}}%
 \newcommand*{\horzbar}{\rule[.5ex]{2.5ex}{0.5pt}}

\global\long\def\revisionmath#1{\textcolor{red}{#1}}%

\global\long\def\revisionmathtwo#1{\textcolor{blue}{#1}}%

\makeatletter \def\ps@pprintTitle{  \let\@oddhead\@empty  \let\@evenhead\@empty  \def\@oddfoot{\centerline{\thepage}}  \let\@evenfoot\@oddfoot} \makeatother

\let\svthefootnote\thefootnote\let\thefootnote\relax\footnotetext{\\ \hspace*{65pt}DISTRIBUTION STATEMENT A. Approved for public release. Distribution is unlimited.}\addtocounter{footnote}{-1}\let\thefootnote\svthefootnote

\section{Introduction\label{sec:Introduction}}

Compressible, multicomponent flows are relevant to many engineering
applications and physical phenomena, including propulsion, atmospheric
entry, and pollutant formation. However, robust and accurate simulation
of such flows remains challenging. For instance, a well-known, long-standing
issue is the failure of conventional conservative numerical schemes
to preserve pressure equilibrium at fluid interfaces in the inviscid
limit, leading to spurious pressure oscillations that can then cause
noticeable errors and even solver divergence. Note that pressure-equilibrium
preservation corresponds to the ability to maintain uniform pressure
under initially constant velocity and pressure in inviscid flows,
and velocity equilibrium is typically implicitly assumed when referring
to pressure equilibrium. These spurious pressure oscillations arise
when the specific heat ratio of the mixture is no longer constant,
which can occur even in the case of a monocomponent thermally perfect
gas or a mixture of calorically perfect gases. Although the addition
of numerical dissipation via, for instance, artificial viscosity or
limiting can mitigate these oscillations and help prevent solver divergence,
these stabilization techniques do not guarantee pressure-equilibrium
preservation and will still often generate errors in pressure~\citep{Chi23_short}.
In addition, such techniques can introduce excessive dissipation (to
the detriment of accuracy) in complex flow problems. To address this
difficulty, many quasi-conservative schemes (i.e., schemes that sacrifice
discrete conservation of mass and/or total energy) have been developed.
One example is the double-flux method~\citep{Abg01}, in which the
thermodynamics are locally frozen to mimic the calorically perfect,
single-species setting. Alternatively, the pressure-evolution equation
can be solved in place of the total-energy equation~\citep{Kar92,Kar94,Ter12,Kaw15}.
These two quasi-conservative approaches mathematically guarantee preservation
of pressure equilibrium at the cost of total-energy conservation.
However, loss of conservation is often considered to be unsatisfactory,
especially at shocks. Hybrid strategies that switch between pressure-equilibrium-preserving
(but quasi-conservative) and conservative (but non-pressure-equilibrium-preserving)
schemes have been proposed~\citep{Kar96,Fed02,Lv15,Boy21,Gab24}.

To prevent loss of conservation, recent efforts in the computational
fluid dynamics (CFD) community have focused on conservative schemes
that mitigate spurious pressure oscillations without relying on artificial
viscosity or limiting. Some of these approaches belong to the family
of numerical schemes known as discontinuous Galerkin (DG) methods~\citep{Ree73,Bas97_2,Bas97,Coc98,Coc00},
which have gained considerable attention due to high-order accuracy,
geometric flexibility, amenability to modern computing systems, and
other advantages~\citep{Wan13}. A number of studies~\citep{Joh20,Ban23,Chi23_short}
have found that the type of integration (i.e., colocation or overintegration)
can have a noticeable effect on the magnitude of spurious pressure
oscillations. In addition, projecting the pressure onto the finite
element test space prior to evaluating the flux (akin to reconstruction
of the primitive variables in finite-volume schemes~\citep{Boy21,Joh06})
can attenuate the pressure oscillations~\citep{Fra16,Joh20,Ban23,Chi23_short}.
 Note that these methods do not mathematically guarantee preservation
of pressure equilibrium. In contrast, Fujiwara et al.~\citep{Fuj23}
recently introduced a fully conservative finite volume scheme for
the compressible, multicomponent Euler equations that provably maintains
pressure equilibrium. The key ingredient was the derivation of a discrete
pressure-equilibrium condition that informed the development of a
provably pressure-equilibrium-preserving numerical flux, albeit with
the caveat that calorically perfect gas mixtures are assumed (i.e.,
each species has a constant specific heat ratio). Terashima et al.~\citep{Ter24}
then extended this approach to real fluids modeled with a cubic equation
of state, although the more complicated relationship between pressure
and internal energy results in an \emph{approximately }pressure-equilibrium-preserving
numerical flux, where second-order spatial errors with respect to
exact pressure equilibrium were reported. To the best of our knowledge,
even in the simpler (but still complex) case of mixtures of thermally
perfect gases (i.e., each species has a temperature-dependent specific
heat ratio), which are considered in this work, conservative schemes
that mathematically preserve pressure equilibrium have remained elusive.
The primary objective of the present study is to address this gap.

We first extend some of the analysis in~\citep{Fuj23} and show that
a standard DG discretization of the conservative form of the multicomponent
Euler equations is velocity-equilibrium-preserving but not pressure-equilibrium-preserving.
We then introduce a pressure-based DG method (i.e., total energy is
replaced with pressure as a state variable) in which semidiscrete
total-energy conservation is restored via the conservative, high-order,
elementwise correction terms of Abgrall~\citep{Abg18} and Abgrall
et al.~\citep{Abg22}. These correction terms enable semidiscrete
satisfaction of auxiliary transport equations (here, the total-energy
equation) in an integral sense. Note that it may seem more natural
to retain total energy as a state variable and design the correction
terms to additionally satisfy the pressure-evolution equation. However,
pressure equilibrium is essentially a pointwise condition, whereas
the correction terms only guarantee integral satisfaction of the auxiliary
transport equation(s). Entropy-conservative/entropy-stable DG schemes
have been constructed using these correction terms (or similar forms)
without relying on SBP operators or entropy-conservative/entropy-stable
numerical fluxes~\citep{Abg18,Che20,Gab23,Man24,Alb24}. The correction
terms have also been employed to enforce conservation of total energy
while treating entropy density as a state variable~\citep{Abg23}.
Here, the difficulties associated with devising a provably pressure-equilibrium-preserving
numerical flux for non-calorically-perfect gases can be circumvented
using the correction terms, although a distinct set of challenges
is also introduced. For instance, although a standard pressure-based
DG method (i.e., without the correction terms) can exactly maintain
pressure and velocity equilibria, the addition of the correction terms
results in loss of this property. Furthermore, for a given element,
the correction terms are only valid if the state inside the element
is non-uniform, thus ignoring the case of elementwise-constant solutions
with inter-element jumps. The correction terms also fail to preserve
zero species concentrations, which can be especially detrimental if
chemical reactions are considered. We develop modifications to the
correction terms that enable exact preservation of pressure equilibrium,
velocity equilibrium, and zero species concentrations (while maintaining
semidiscrete total-energy conservation) in multicomponent flows, although
it should be noted that there are certain tradeoffs. In addition,
we propose combining the elementwise correction terms with face-based
corrections of the form presented in~\citep{Abg23} in order to account
for elementwise-constant solutions with inter-element jumps. Detailed
comparisons of the correction terms with and without the modifications
are performed.  In this study, we target smooth, interfacial flows
initially in pressure and velocity equilibria, as in~\citep{Fuj23}
and~\citep{Ter24}, and mixtures of thermally perfect gases. No artificial
viscosity or limiting is applied in the considered test cases. The
eventual goal is to account for flows with or without discontinuities
involving real-fluid mixtures, which are even more susceptible to
large-scale spurious pressure oscillations and solver divergence due
to the additional thermodynamic nonlinearities. This will be the subject
of future work.

Note that the proposed formulation is similar in spirit to a recent
ADER-DG method developed by Gaburro et al.~\citep{Gab24}, which
also treats pressure as a state variable and incorporates total-energy-based
corrections. However, there are some key differences. First, the total-energy
corrections in~\citep{Gab24} are of a subcell finite-volume type,
distinct from the correction terms employed in this study. In addition,
the ADER-DG method in~\citep{Gab24} only applies the corrections
at shocks, which are detected using a shock sensor, such that total
energy is not conserved everywhere in the domain. It should be noted
that total-energy conservation is most critical at shocks~\citep{Kar92,Kar94,Cas08_2,Abg10}
and may not need to be satisfied elsewhere to obtain accurate results.
Finally, the subcell finite-volume-type corrections, if applied to
a constant-pressure fluid interface, may not preserve pressure equilibrium.
Regardless, the method in~\citep{Gab24} was used to obtain very
encouraging results in discontinuous multi-material flows involving
calorically perfect gases. Here, we specifically focus on recovering
total-energy conservation everywhere in the domain; a detailed investigation
of whether total-energy conservation can be sacrificed away from shocks,
especially in practical problems, is outside the scope of the current
study but may be pursued in the future. 

The remainder of this paper is organized as follows. Section~\ref{sec:energy-based-formulation}
briefly summarizes the conservative form of the governing equations
and the corresponding DG discretization. The pressure-based formulation,
correction terms, and proposed modifications are introduced in the
following section. Results for a variety of smooth, interfacial flow
problems initially in pressure and velocity equilibria are given in
Section~\ref{sec:results}, wherein curved elements are also considered.
The paper concludes with some broader discussion and final remarks.

\section{Compressible, multicomponent Euler equations: Total-energy-based
formulation}

\label{sec:energy-based-formulation}

Before discussing the pressure-based formulation, we begin with the
conservative form of the compressible, multicomponent Euler equations:
\begin{equation}
\partial_{t}\bm{y}+\nabla\cdot\bm{\mathcal{F}}\left(\bm{y}\right)=0,\label{eq:conservation-law-strong-form}
\end{equation}
where $t$ is the time, $\bm{y}$ is the vector of $m$ state variables,
and $\bm{\mathcal{F}}$ is the convective flux. The physical coordinates
are denoted by $x=(x_{1},\ldots,x_{d})$, where $d$ is the number
of spatial dimensions. The state vector is expanded as
\begin{equation}
\bm{y}=\left(\rho v_{1},\ldots,\rho v_{d},\rho e_{t},C_{1},\ldots,C_{n_{s}}\right)^{T},\label{eq:reacting-navier-stokes-state}
\end{equation}
where $\rho$ is the density, $\bm{v}=\left(v_{1},\ldots,v_{d}\right)$
is the velocity vector, $e_{t}$ is the specific total energy, and
$C_{i}$ is the molar concentration of the $i$th (out of $n_{s}$)
species. Note that the first $d$ components of $\bm{y}$ correspond
to momentum, the $\left(d+1\right)$th component corresponds to energy,
and the last $n_{s}$ components correspond to species concentrations.
The density is computed from the species concentrations as
\[
\rho=\sum_{i=1}^{n_{s}}\rho_{i}=\sum_{i=1}^{n_{s}}W_{i}C_{i},
\]
where $\rho_{i}$ and $W_{i}$ are the partial density and molar mass,
respectively, of the $i$th species. Throughout this work, we assume
$\rho>0$ (i.e., no vacuum). The mass and mole fractions of the $i$th
species are given by
\[
Y_{i}=\frac{\rho_{i}}{\rho},\quad X_{i}=\frac{C_{i}}{\sum_{i=1}^{n_{s}}C_{i}},
\]
and the specific total energy is expanded as 
\[
e_{t}=u+\frac{1}{2}\sum_{k=1}^{d}v_{k}v_{k},
\]
where $u=\sum_{i=1}^{n_{s}}Y_{i}u_{i}$ is the mixture-averaged specific
internal energy. Under the assumption of thermally perfect gases,
$u_{i}$ is defined as~\citep{Gio99}
\[
u_{i}=h_{i}-R_{i}T=h_{\mathrm{ref},i}+\int_{T_{\mathrm{ref}}}^{T}c_{p,i}(\tau)d\tau-R_{i}T,
\]
where $h_{i}$ is the mass-specific enthalpy of the $i$th species,
$R_{i}=R^{0}/W_{i}$ (with $R^{0}$ denoting the universal gas constant),
$T$ is the temperature, $T_{\mathrm{ref}}=298.15\:\mathrm{K}$ is
the reference temperature, $h_{\mathrm{ref},i}$ is the reference-state
species formation enthalpy, and $c_{p,i}$ is the mass-specific heat
capacity at constant pressure of the $i$th species. $c_{p,i}$ is
computed from an $n_{p}$-order polynomial as
\begin{equation}
c_{p,i}=\sum_{k=0}^{n_{p}}a_{ik}T^{k},\label{eq:specific_heat_polynomial}
\end{equation}
based on the NASA thermodynamic fits~\citep{Mcb93,Mcb02}.

The $k$th spatial component of the convective flux is given by
\begin{equation}
\bm{\mathcal{F}}_{k}\left(y\right)=\left(\rho v_{k}v_{1}+P\delta_{k1},\ldots,\rho v_{k}v_{d}+P\delta_{kd},v_{k}\left(\rho e_{t}+P\right),v_{k}C_{1},\ldots,v_{k}C_{n_{s}}\right)^{T},\label{eq:reacting-navier-stokes-spatial-convective-flux-component-1}
\end{equation}
where $P$ is the pressure, which is computed from the ideal-gas law:
\begin{equation}
P=R^{0}T\sum_{i}C_{i}=\rho RT.\label{eq:eos}
\end{equation}
$R=R^{0}/\overline{W}$ is the specific gas constant, where $\overline{W}=\rho/\sum_{i}^{n_{s}}C_{i}$
is the molar mass of the mixture.

\subsection{Discontinuous Galerkin discretization\label{sec:DG-discretization}}

Let $\Omega$ denote the computational domain partitioned by $\mathcal{T}$,
which consists of cells $\kappa$ with boundaries $\partial\kappa$.
Let $\mathcal{E}$ denote the set of interfaces $\epsilon$, consisting
of the interior interfaces,
\[
\epsilon_{\mathcal{I}}\in\mathcal{E_{I}}=\left\{ \epsilon_{\mathcal{I}}\in\mathcal{E}\middlebar\epsilon_{\mathcal{I}}\cap\partial\Omega=\emptyset\right\} ,
\]
and boundary interfaces, 
\[
\epsilon_{\partial}\in\mathcal{E}_{\partial}=\left\{ \epsilon_{\partial}\in\mathcal{E}\middlebar\epsilon_{\partial}\subset\partial\Omega\right\} .
\]
At interior interfaces, there exist $\kappa^{+}$ and $\kappa^{-}$
such that $\epsilon_{\mathcal{I}}=\partial\kappa^{+}\cap\partial\kappa^{-}$.
$n^{+}$ and $n^{-}$ denote the outward facing normals of $\kappa^{+}$
and $\kappa^{-}$, respectively. Let $V_{h}^{p}$ denote the space
of test functions,
\begin{eqnarray}
V_{h}^{p} & = & \left\{ \bm{\mathfrak{v}}\in\left[L^{2}\left(\Omega\right)\right]^{m}\middlebar\forall\kappa\in\mathcal{T},\left.\bm{\mathfrak{v}}\right|_{\kappa}\in\left[\mathcal{P}_{p}(\kappa)\right]^{m}\right\} ,\label{eq:discrete-subspace}
\end{eqnarray}
where $\mathcal{P}_{p}(\kappa)$ is a space of polynomial functions
of degree no greater than $p$ in $\kappa$.

The semi-discrete form of Equation~(\ref{eq:conservation-law-strong-form})
is as follows: find $\bm{y}\in V_{h}^{p}$ such that
\begin{gather}
\sum_{\kappa\in\mathcal{T}}\left(\partial_{t}\bm{y},\bm{\mathfrak{v}}\right)_{\kappa}-\sum_{\kappa\in\mathcal{T}}\left(\bm{\mathcal{F}}\left(\bm{y}\right),\nabla\bm{\mathfrak{v}}\right)_{\kappa}+\sum_{\epsilon\in\mathcal{E}}\left(\bm{\mathcal{F}}^{\dagger}\left(\bm{y},\bm{n}\right),\left\llbracket \bm{\mathfrak{v}}\right\rrbracket \right)_{\mathcal{E}}=0\qquad\forall\:\bm{\mathfrak{v}}\in V_{h}^{p},\label{eq:semi-discrete-form}
\end{gather}
where $\left(\cdot,\cdot\right)$ denotes the inner product, $\bm{\mathcal{F}}^{\dagger}\left(y,n\right)$
is the flux function, and $\left\llbracket \cdot\right\rrbracket $
is the jump operator. At interior interfaces, $\left\llbracket \bm{\mathfrak{v}}\right\rrbracket =\bm{\mathfrak{v}}^{+}-\bm{\mathfrak{v}}^{-}$
and $\bm{\mathcal{F}}^{\dagger}\left(\bm{y},\bm{n}\right)=\bm{\mathcal{F}}^{\dagger}\left(\bm{y}^{+},\bm{y}^{-},\bm{n}\right)$,
where $\bm{\mathcal{F}}^{\dagger}\left(\bm{y}^{+},\bm{y}^{-},\bm{n}\right)$
is a numerical flux. Throughout this work, we use the local Lax-Friedrichs
numerical flux,
\begin{align}
\bm{\mathcal{F}}^{\dagger}\left(\bm{y}^{+},\bm{y}^{-},\bm{n}\right) & =\average{\bm{\mathcal{F}}\left(\bm{y}\right)}\cdot\bm{n}+\frac{1}{2}\lambda\left\llbracket \bm{y}\right\rrbracket ,\label{eq:lax-friedrichs}
\end{align}
where $\average{\cdot}$ denotes the average operator and $\lambda$
is an estimate of the maximum wave speed across $\bm{y}^{+}$ and
$\bm{y}^{-}$. More information on boundary conditions can be found
in~\citep{Joh20,Chi24_viscous}. The element-local solution, $\left.\bm{y}\right|_{\kappa}$,
can be expanded as (dropping the $\left.\left(\cdot\right)\right|_{\kappa}$
notation for brevity)
\[
\bm{y}=\sum_{j=1}^{n_{b}}\widehat{\bm{y}}_{j}\phi_{j},
\]
where $\widehat{\bm{y}}_{j}$ is the $j$th polynomial coefficient,
$\phi_{j}$ is the $j$th basis function, and $n_{b}$ is the number
of basis functions. A nodal basis is employed in this work. In element-local
integral form, we have (assuming all interior faces)

\begin{equation}
\int_{\kappa}\phi_{i}\partial_{t}\bm{y}dx-\int_{\kappa}\nabla\phi_{i}\cdot\bm{\mathcal{F}}\left(\bm{y}\right)dx+\oint_{\partial\kappa}\phi_{i}\bm{\mathcal{F}}^{\dagger}\left(\bm{y}^{+},\bm{y}^{-},\bm{n}\right)ds=0,\quad i=1,\ldots,n_{b}.\label{eq:dg-integral-form-conservative}
\end{equation}
Introducing the mass matrix, $\mathcal{M}_{ij}=\int_{\kappa}\phi_{i}\phi_{j}dx,$
we can rewrite Equation~(\ref{eq:dg-integral-form-conservative})
as
\begin{equation}
\bm{\mathcal{M}}d_{t}\widehat{\bm{y}}-\int_{\kappa}\nabla\bm{\Phi}\cdot\bm{\mathcal{F}}\left(\bm{y}\right)dx+\oint_{\partial\kappa}\bm{\Phi}\bm{\mathcal{F}}^{\dagger}\left(\bm{y}^{+},\bm{y}^{-},\bm{n}\right)ds=0,\label{eq:dg-mass-matrix-form-conservative}
\end{equation}
where $\widehat{\bm{y}}=\left[\widehat{\bm{y}}_{1},\ldots,\widehat{\bm{y}}_{n_{b}}\right]$
is the vector of polynomial coefficients and $\bm{\Phi}=\left[\phi_{1},\ldots,\phi_{n_{b}}\right]^{T}$
is the vector of basis functions. Left-multiplying both sides by $\bm{\Phi}^{T}\bm{\mathcal{M}}^{-1}$
yields
\begin{equation}
\partial_{t}\bm{y}-\bm{\Phi}^{T}\bm{\mathcal{M}}^{-1}\int_{\kappa}\nabla\bm{\Phi}\cdot\bm{\mathcal{F}}\left(\bm{y}\right)dx+\bm{\Phi}^{T}\bm{\mathcal{M}}^{-1}\oint_{\partial\kappa}\bm{\Phi}\bm{\mathcal{F}}^{\dagger}\left(\bm{y}^{+},\bm{y}^{-},\bm{n}\right)ds=0,\label{eq:dg-state-evolution-conservative}
\end{equation}
which represents the temporal evolution of $\bm{y}\left(\bm{x},t\right)$.
Here, it is implicitly assumed that, for example, $\bm{\mathcal{M}}d_{t}\widehat{\bm{y}}$
represents the matrix-vector product between the mass matrix and each
of the $m$ components of $d_{t}\widehat{\bm{y}}$ (and similarly
for related operations).

\subsubsection{Velocity-equilibrium preservation}

We examine whether the DG discretization in the previous subsection
exactly preserves velocity equilibrium. This represents an extension
of the analysis by Fujiwara et al.~\citep{Fuj23} to DG schemes and
a generalization of that by~\citep{Ban23}, who focused on discontinuous
interfaces, to smooth interfaces. The velocity evolution equation
is given by~\citep{Fuj23}

\begin{equation}
\partial_{t}\bm{v}=\frac{1}{\rho}\partial_{t}\left(\rho\bm{v}\right)-\frac{\bm{v}}{\rho}\partial_{t}\rho,\label{eq:velocity-evolution-equation}
\end{equation}
where $\partial_{t}\rho=\sum_{i}W_{i}\partial_{t}C_{i}$. For simplicity,
we focus on the one-dimensional case, but extensions to $d>1$ are
straightforward. 
\begin{lem}
\label{lem:velocity-equilibrium-preservation-standard-dg}In the case
of constant pressure $(P=P_{0})$ and velocity $\left(\bm{v}=\bm{v}_{0}\right)$,
the DG discretization~(\ref{eq:dg-integral-form-conservative}) preserves
velocity equilibrium (i.e., $\partial_{t}v=0$ necessarily) if and
only if the numerical flux satisfies
\begin{equation}
P_{0}\oint_{\partial\kappa}\bm{\Phi}\cdot ndx-\oint_{\partial\kappa}\bm{\Phi}\mathcal{F}_{\rho v}^{\dagger}\left(\bm{y}^{+},\bm{y}^{-},n\right)ds+v_{0}\oint_{\partial\kappa}\bm{\Phi}\sum_{i}W_{i}\mathcal{F}_{C_{i}}^{\dagger}\left(\bm{y}^{+},\bm{y}^{-},n\right)ds=0,\label{eq:velocity-equilibrium-condition-general-numerical-flux}
\end{equation}
where $\mathcal{F}_{\rho v}^{\dagger}\left(\bm{y}^{+},\bm{y}^{-},n\right)$
is the momentum component of the numerical flux and $\mathcal{F}_{C_{i}}^{\dagger}\left(\bm{y}^{+},\bm{y}^{-},n\right)$
is the component corresponding to the $i$th species concentration,.
\end{lem}

\begin{proof}
Substituting the momentum and concentration components of the temporal
evolution of the local solution~(\ref{eq:dg-state-evolution-conservative})
into Equation~(\ref{eq:velocity-evolution-equation}) yields
\begin{align*}
\partial_{t}v= & \bm{\Phi}^{T}\bm{\mathcal{M}}^{-1}\left(\frac{1}{\rho}\left[\int_{\kappa}\partial_{x}\bm{\Phi}\begin{pmatrix}\rho v^{2}+P\end{pmatrix}dx-\oint_{\partial\kappa}\bm{\Phi}\mathcal{F}_{\rho v}^{\dagger}\left(\bm{y}^{+},\bm{y}^{-},n\right)ds\right]\right.\\
 & \left.-\frac{v}{\rho}\left[\int_{\kappa}\partial_{x}\bm{\Phi}\begin{pmatrix}\rho v\end{pmatrix}dx-\oint_{\partial\kappa}\bm{\Phi}\sum_{i}W_{i}\mathcal{F}_{C_{i}}^{\dagger}\left(\bm{y}^{+},\bm{y}^{-},n\right)ds\right]\right)\\
= & \bm{\Phi}^{T}\bm{\mathcal{M}}^{-1}\Biggl(\frac{1}{\rho}\left[v_{0}^{2}\int_{\kappa}\partial_{x}\bm{\Phi}\rho dx+P_{0}\int_{\kappa}\partial_{x}\bm{\Phi}dx-\oint_{\partial\kappa}\bm{\Phi}\mathcal{F}_{\rho v}^{\dagger}\left(\bm{y}^{+},\bm{y}^{-},n\right)ds\right]\\
 & -\frac{v_{0}}{\rho}\left[v_{0}\int_{\kappa}\partial_{x}\bm{\Phi}\cdot\begin{pmatrix}\rho\end{pmatrix}dx-\oint_{\partial\kappa}\bm{\Phi}\sum_{i}W_{i}\mathcal{F}_{C_{i}}^{\dagger}\left(\bm{y}^{+},\bm{y}^{-},n\right)ds\right]\Biggr)\\
= & \bm{\Phi}^{T}\bm{\mathcal{M}}^{-1}\left(\frac{P_{0}}{\rho}\oint_{\partial\kappa}\bm{\Phi}\cdot ndx-\frac{1}{\rho}\oint_{\partial\kappa}\bm{\Phi}\mathcal{F}_{\rho v}^{\dagger}\left(\bm{y}^{+},\bm{y}^{-},n\right)ds+\frac{v_{0}}{\rho}\oint_{\partial\kappa}\bm{\Phi}\sum_{i}W_{i}\mathcal{F}_{C_{i}}^{\dagger}\left(\bm{y}^{+},\bm{y}^{-},n\right)ds\right),
\end{align*}
where the last line assumes sufficiently accurate numerical integration.
Clearly, if~(\ref{eq:velocity-equilibrium-condition-general-numerical-flux})
is satisfied, the RHS vanishes and $\partial_{t}v=0$. Next, we consider
the reverse. Since $\left\{ \phi_{1},\ldots,\phi_{n_{b}}\right\} $
forms a basis of $\mathcal{P}_{p}(\kappa)$, $\partial_{t}v=0$ requires
\[
\bm{\mathcal{M}}^{-1}\left(\frac{P_{0}}{\rho}\oint_{\partial\kappa}\bm{\Phi}\cdot ndx-\frac{1}{\rho}\oint_{\partial\kappa}\bm{\Phi}\mathcal{F}_{\rho v}^{\dagger}\left(\bm{y}^{+},\bm{y}^{-},n\right)ds+\frac{v_{0}}{\rho}\oint_{\partial\kappa}\bm{\Phi}\sum_{i}W_{i}\mathcal{F}_{C_{i}}^{\dagger}\left(\bm{y}^{+},\bm{y}^{-},n\right)ds\right)=0,
\]
which implies that the quantity inside the parentheses is zero because
the mass matrix is nonsingular. Therefore, (\ref{eq:velocity-equilibrium-condition-general-numerical-flux})
must be satisfied.
\end{proof}
\begin{rem}
The Lax-Friedrichs flux~(\ref{eq:lax-friedrichs}) satisfies~(\ref{eq:velocity-equilibrium-condition-general-numerical-flux}).
To show this, we substitute~(\ref{eq:lax-friedrichs}) into~(\ref{eq:velocity-equilibrium-condition-general-numerical-flux}),
yielding
\[
\frac{P_{0}}{\rho}\oint_{\partial\kappa}\bm{\Phi}\cdot ndx-\frac{1}{\rho}\oint_{\partial\kappa}\bm{\Phi}\left[v_{0}^{2}\average{\rho}\cdot n+P_{0}\cdot n+\frac{v_{0}}{2}\lambda\left\llbracket \rho\right\rrbracket \right]ds+\frac{v_{0}}{\rho}\oint_{\partial\kappa}\bm{\Phi}\left[v_{0}\average{\rho}\cdot n+\frac{1}{2}\lambda\left\llbracket \rho\right\rrbracket \right]ds=0
\]
since all terms on the LHS cancel out. Other numerical fluxes, such
as the HLLC flux~\citep{Tor13}, also satisfy~(\ref{eq:velocity-equilibrium-condition-general-numerical-flux}).
\end{rem}

\begin{rem}
The above analysis focuses on the semidiscrete setting. For certain
time integrators, it is straightforward to show fully discrete preservation
of velocity equilibrium. For example, letting $\left(\cdot\right)^{n}$
denote the $n$th time step, the Forward Euler scheme combined with
the Lax-Friedrichs numerical flux gives
\begin{align*}
\left(\rho\right)^{n+1} & =\left(\rho\right)^{n}+\bm{\Phi}^{T}\bm{\mathcal{M}}^{-1}\Delta t\left[\int_{\kappa}\partial_{x}\bm{\Phi}\begin{pmatrix}\rho v\end{pmatrix}dx-\oint_{\partial\kappa}\bm{\Phi}\sum_{i}W_{i}\mathcal{F}_{C_{i}}^{\dagger}\left(\bm{y}^{+},\bm{y}^{-},n\right)ds\right]^{n}\\
 & =\left(\rho\right)^{n}+\bm{\Phi}^{T}\bm{\mathcal{M}}^{-1}\Delta t\left[v_{0}\int_{\kappa}\partial_{x}\bm{\Phi}\rho dx-\oint_{\partial\kappa}\bm{\Phi}\sum_{i}W_{i}\mathcal{F}_{C_{i}}^{\dagger}\left(\bm{y}^{+},\bm{y}^{-},n\right)ds\right]^{n}
\end{align*}
and
\begin{align*}
\left(\rho v\right)^{n+1} & =\left(\rho v\right)^{n}+\bm{\Phi}^{T}\bm{\mathcal{M}}^{-1}\Delta t\left[\int_{\kappa}\partial_{x}\bm{\Phi}\begin{pmatrix}\rho v^{2}+P\end{pmatrix}dx-\oint_{\partial\kappa}\bm{\Phi}\mathcal{F}_{\rho v}^{\dagger}\left(\bm{y}^{+},\bm{y}^{-},n\right)ds\right]^{n}\\
 & =\left(\rho v\right)^{n}+\bm{\Phi}^{T}\bm{\mathcal{M}}^{-1}\Delta t\left[v_{0}^{2}\int_{\kappa}\partial_{x}\bm{\Phi}\rho dx+P_{0}\oint_{\partial\kappa}\bm{\Phi}\cdot ndx-\oint_{\partial\kappa}\bm{\Phi}\mathcal{F}_{\rho v}^{\dagger}\left(\bm{y}^{+},\bm{y}^{-},n\right)ds\right]^{n}.
\end{align*}
Assuming~(\ref{eq:velocity-equilibrium-condition-general-numerical-flux})
holds, we can rewrite $\left(\rho v\right)^{n+1}$ as
\begin{align*}
\left(\rho v\right)^{n+1} & =\left(\rho v\right)^{n}+\bm{\Phi}^{T}\bm{\mathcal{M}}^{-1}\Delta t\left[v_{0}^{2}\int_{\kappa}\partial_{x}\bm{\Phi}\rho dx-v_{0}\oint_{\partial\kappa}\bm{\Phi}\sum_{i}W_{i}\mathcal{F}_{C_{i}}^{\dagger}\left(\bm{y}^{+},\bm{y}^{-},n\right)ds\right]^{n}\\
 & =v_{0}\left(\rho\right)^{n}+\bm{\Phi}^{T}\bm{\mathcal{M}}^{-1}\Delta tv_{0}\left[v_{0}\int_{\kappa}\partial_{x}\bm{\Phi}\rho dx-\oint_{\partial\kappa}\bm{\Phi}\sum_{i}W_{i}\mathcal{F}_{C_{i}}^{\dagger}\left(\bm{y}^{+},\bm{y}^{-},n\right)ds\right]^{n}\\
 & =v_{0}\left(\rho\right)^{n+1}.
\end{align*}
Therefore, 
\[
v^{n+1}=\frac{\left(\rho v\right)^{n+1}}{\left(\rho\right)^{n+1}}=\frac{v_{0}\left(\rho\right)^{n+1}}{\left(\rho\right)^{n+1}}=v_{0}.
\]
This result directly extends to strong-stability-preserving Runge-Kutta
(SSPRK) schemes~\citep{Got01}, which can be expressed as convex
combinations of Forward Euler steps.
\end{rem}

\begin{rem}
\label{rem:exact-integration-of-pressure}In the above analysis, it
is assumed that $\int_{\kappa}\partial_{x}\bm{\Phi}Pdx$ is exactly
integrated (i.e., $\int_{\kappa}\partial_{x}\bm{\Phi}Pdx$ is evaluated
to be equal to $P_{0}\int_{\kappa}\partial_{x}\bm{\Phi}dx=P_{0}\oint_{\partial\kappa}\bm{\Phi}\cdot ndx$).
However, with standard overintegration, this will often not be the
case due to the nonlinear relationship between pressure and internal
energy. Assuming a nodal basis and the nodal values of pressure to
be $P_{0}$, colocated integration, as well as an overintegration
strategy wherein pressure is projected onto $\mathcal{P}_{p}(\kappa)$
in a particular manner, $\int_{\kappa}\partial_{x}\bm{\Phi}Pdx$ will
be exactly integrated. See~\citep{Joh20,Ban23} for additional discussion.
Regardless, failure to additionally preserve pressure equilibrium,
which will be demonstrated next, will eventually disturb velocity
equilibrium.
\end{rem}

\begin{rem}
The above analysis holds regardless of the thermodynamic model (e.g.,
calorically perfect, thermally perfect, non-ideal, etc.).
\end{rem}

\subsubsection{Pressure-equilibrium preservation}

\label{subsec:Pressure-equilibrium-preservation}

We now turn to pressure-equilibrium preservation. With $\rho u=\rho u\left(P,C_{1},\ldots,C_{n_{s}}\right)$,
the pressure-evolution equation can be written in terms of the conservative
variables as~\citep{Fuj23}

\begin{equation}
\partial_{t}P=\left(\frac{\partial\rho u}{\partial P}\right)_{C_{i}}^{-1}\left[\partial_{t}\left(\rho e_{t}\right)-v\partial_{t}\left(\rho v\right)+\frac{v^{2}}{2}\partial_{t}\rho-\sum_{i}\left(\frac{\partial\rho u}{\partial C_{i}}\right)_{P,C_{j}\neq i}\partial_{t}C_{i}\right].\label{eq:pressure-evolution-equation-conservative-variables}
\end{equation}

\begin{prop}
In the case of constant pressure $(P=P_{0})$ and velocity $\left(v=v_{0}\right)$,
the DG discretization~(\ref{eq:dg-integral-form-conservative}) preserves
pressure equilibrium (i.e., $\partial_{t}P=0$ necessarily) if and
only if the condition
\begin{equation}
\begin{aligned} & v_{0}\int_{\kappa}\partial_{x}\bm{\Phi}\rho udx-\oint_{\partial\kappa}\bm{\Phi}\mathcal{F}_{\rho e_{t}}^{\dagger}\left(\bm{y}^{+},\bm{y}^{-},n\right)ds+v_{0}\oint_{\partial\kappa}\bm{\Phi}\mathcal{F}_{\rho v}^{\dagger}\left(\bm{y}^{+},\bm{y}^{-},n\right)ds\\
 & -\frac{v_{0}^{2}}{2}\oint_{\partial\kappa}\bm{\Phi}\sum_{i}W_{i}\mathcal{F}_{C_{i}}^{\dagger}\left(\bm{y}^{+},\bm{y}^{-},n\right)ds=\sum_{i}\left(\frac{\partial\rho u}{\partial\rho_{i}}\right)_{P,\rho_{j}\neq i}\left[v_{0}\int_{\kappa}\partial_{x}\bm{\Phi}\rho_{i}dx-\oint_{\partial\kappa}\bm{\Phi}W_{i}\mathcal{F}_{C_{i}}^{\dagger}\left(\bm{y}^{+},\bm{y}^{-},n\right)ds\right]
\end{aligned}
\label{eq:pressure-equilibrium-condition-general-numerical-flux}
\end{equation}
is satisfied, where $\mathcal{F}_{\rho e_{t}}^{\dagger}\left(y^{+},y^{-},n\right)$
is the total-energy component of the numerical flux. 
\end{prop}

\begin{proof}
Substituting the temporal evolution of the local solution~(\ref{eq:dg-state-evolution-conservative})
into Equation~(\ref{eq:pressure-evolution-equation-conservative-variables})
yields
\begin{align*}
\frac{\partial P}{\partial t}= & \bm{\Phi}^{T}\bm{\mathcal{M}}^{-1}\left(\frac{\partial\rho u}{\partial P}\right)_{C_{i}}^{-1}\Biggl\{ v_{0}\int_{\kappa}\partial_{x}\bm{\Phi}\rho udx-\oint_{\partial\kappa}\bm{\Phi}\mathcal{F}_{\rho e_{t}}^{\dagger}\left(\bm{y}^{+},\bm{y}^{-},n\right)ds\\
 & +v_{0}\oint_{\partial\kappa}\bm{\Phi}\mathcal{F}_{\rho v}^{\dagger}\left(\bm{y}^{+},\bm{y}^{-},n\right)ds-\frac{v_{0}^{2}}{2}\oint_{\partial\kappa}\bm{\Phi}\sum_{i}W_{i}\mathcal{F}_{C_{i}}^{\dagger}\left(\bm{y}^{+},\bm{y}^{-},n\right)ds\\
 & -\sum_{i}\left(\frac{\partial\rho u}{\partial C_{i}}\right)_{P,\rho_{j}\neq i}\left[v_{0}\int_{\kappa}\partial_{x}\bm{\Phi}\rho_{i}dx-\oint_{\partial\kappa}\bm{\Phi}W_{i}\mathcal{F}_{C_{i}}^{\dagger}\left(\bm{y}^{+},\bm{y}^{-},n\right)ds\right]\Biggr\}.
\end{align*}
Assuming $\rho>0$, $\left(\frac{\partial\rho u}{\partial P}\right)_{C_{i}}^{-1}=\left(\frac{\partial\rho e_{t}}{\partial P}\right)_{C_{i}}^{-1}>0$
(see~\ref{sec:appendix-derivative-total-energy}). The proof then
follows the same logic as the proof of Lemma~\ref{lem:velocity-equilibrium-preservation-standard-dg}.
\end{proof}
\begin{rem}
(\ref{eq:pressure-equilibrium-condition-general-numerical-flux})
is an extension of the discrete-pressure-equilibrium condition by
Fujiwara et al.~\citep{Fuj23} from finite-volume methods to DG schemes,
which is complicated further by the presence of the volumetric terms.
Even restricting the discussion to a finite-volume framework, most
numerical fluxes, including the Lax-Friedrichs flux, generally do
not satisfy the discrete-pressure-equilibrium condition~\citep{Fuj23}.
Fujiwara et al.~\citep{Fuj23} developed a numerical flux that is
provably pressure-equilibrium-preserving for the specific case of
mixtures of calorically perfect gases. Terashima et al.~\citep{Ter24}
extended this numerical flux to real-fluid mixtures, although the
more complicated relationship between pressure and internal energy
allows for only \emph{approximate} pressure-equilibrium preservation.
\end{rem}

\begin{rem}
In the case of a monocomponent calorically perfect gas, 
\[
\rho u=\frac{P}{\gamma-1},\quad\left(\frac{\partial\rho u}{\partial P}\right)_{C}=\frac{1}{\gamma-1},\quad\left(\frac{\partial\rho u}{\partial C}\right)_{P}=0,
\]
where $C$ is the concentration of the sole species and $\gamma$
is the (constant) specific heat ratio. The pressure-equilibrium condition~(\ref{eq:pressure-equilibrium-condition-general-numerical-flux})
then reduces to
\begin{equation}
\frac{v_{0}P_{0}}{\gamma-1}\int_{\kappa}\partial_{x}\bm{\Phi}dx-\oint_{\partial\kappa}\bm{\Phi}\mathcal{F}_{\rho e_{t}}^{\dagger}\left(\bm{y}^{+},\bm{y}^{-},n\right)ds+v_{0}\oint_{\partial\kappa}\bm{\Phi}\mathcal{F}_{\rho v}^{\dagger}\left(\bm{y}^{+},\bm{y}^{-},n\right)ds-\frac{v_{0}^{2}}{2}W\oint_{\partial\kappa}\bm{\Phi}\mathcal{F}_{C}^{\dagger}\left(\bm{y}^{+},\bm{y}^{-},n\right)ds=0.\label{eq:pressure-equilibrium-condition-general-numerical-flux-CPG}
\end{equation}
Substituting the definition of the Lax-Friedrichs flux function into
the LHS of~(\ref{eq:pressure-equilibrium-condition-general-numerical-flux-CPG})
yields
\[
\begin{aligned} & \frac{v_{0}P_{0}}{\gamma-1}\int_{\kappa}\partial_{x}\bm{\Phi}dx-\oint_{\partial\kappa}\bm{\Phi}\left[v_{0}\left(\frac{P_{0}}{\gamma-1}+\frac{1}{2}v_{0}^{2}\average{\rho}+P_{0}\right)\cdot n+\frac{1}{2}\lambda\left\llbracket \rho e_{t}\right\rrbracket \right]ds\\
 & +v_{0}\oint_{\partial\kappa}\bm{\Phi}\left[\left(v_{0}^{2}\average{\rho}+P_{0}\right)\cdot n+\frac{v_{0}}{2}\lambda\left\llbracket \rho\right\rrbracket \right]ds-\frac{v_{0}^{2}}{2}\oint_{\partial\kappa}\bm{\Phi}\left[v_{0}\average{\rho}\cdot n+\frac{1}{2}\lambda\left\llbracket \rho\right\rrbracket \right]ds\\
 & =\frac{v_{0}P_{0}}{\gamma-1}\oint_{\partial\kappa}\bm{\Phi}ds-\frac{v_{0}P_{0}}{\gamma-1}\oint_{\partial\kappa}\bm{\Phi}ds-\frac{1}{2}\oint_{\partial\kappa}\bm{\Phi}\lambda\left\llbracket \frac{P_{0}}{\gamma-1}\right\rrbracket ds-\frac{v_{0}^{2}}{4}\oint_{\partial\kappa}\bm{\Phi}\lambda\left\llbracket \rho\right\rrbracket ds+\frac{v_{0}^{2}}{4}\oint_{\partial\kappa}\bm{\Phi}\lambda\left\llbracket \rho\right\rrbracket ds,
\end{aligned}
\]
which vanishes; therefore, pressure equilibrium is discretely preserved.
\end{rem}

\section{Compressible, multicomponent Euler equations: Pressure-based formulation}

We now introduce a different strategy wherein the total-energy equation
is replaced by the pressure-evolution equation. This normally results
in exact preservation of pressure equilibrium but loss of total-energy
conservation; therefore, we incorporate the correction terms of Abgrall~\citep{Abg18}
and Abgrall et al.~\citep{Abg22} to restore total-energy conservation.
However, additional modifications are required to retain exact preservation
of pressure equilibrium, velocity equilibrium, and zero species concentrations.

The pressure-evolution equation is given by~\citep{Ter12,Kaw15}
\begin{equation}
\partial_{t}P+\nabla\cdot\left(P\bm{v}\right)+\left(\rho c^{2}-P\right)\nabla\cdot\bm{v}=0,\label{eq:pressure-evolution-equation}
\end{equation}
where $c$ is the speed of sound. Note that for thermally perfect
gases, $c^{2}=\gamma P/\rho$, where $\gamma$ is the specific heat
ratio, such that Equation~(\ref{eq:pressure-evolution-equation})
reduces to
\[
\partial_{t}P+\nabla\cdot\left(P\bm{v}\right)+\left(\gamma-1\right)P\nabla\cdot\bm{v}=0.
\]
With the total-energy equation replaced by the pressure-evolution
equation, the governing equations are now written in nonconservative
form as
\begin{equation}
\partial_{t}\bm{y}+\nabla\cdot\bm{\mathcal{F}}\left(\bm{y}\right)+\bm{\mathcal{B}}\left(\bm{y}\right):\nabla\bm{y}=0,\label{eq:conservation-law-strong-form-with-pressure}
\end{equation}
where the state vector is now defined as
\begin{equation}
\bm{y}=\left(\rho v_{1},\ldots,\rho v_{d},P,C_{1},\ldots,C_{n_{s}}\right)^{T}\label{eq:reacting-navier-stokes-state-pressure}
\end{equation}
and the $k$th spatial component of $\bm{\mathcal{F}}\left(\bm{y}\right)$
is given by
\[
\bm{\mathcal{F}}_{k}\left(\bm{y}\right)=\left(\rho v_{k}v_{1}+P\delta_{k1},\ldots,\rho v_{k}v_{d}+P\delta_{kd},Pv_{k},v_{k}C_{1},\ldots,v_{k}C_{n_{s}}\right)^{T}.
\]
$\bm{\mathcal{B}}\left(\bm{y}\right)$ is a third-order tensor for
which 
\[
\bm{\mathcal{B}}\left(\bm{y}\right):\nabla\bm{y}=\left(0,\ldots,0,\left(\rho c^{2}-P\right)\nabla\cdot\bm{v},0,\ldots,0\right)^{T}.
\]
The only nonzero component of $\bm{\mathcal{B}}\left(y\right)$ is
that corresponding to pressure, denoted $\bm{\mathcal{B}}_{P}\left(\bm{y}\right)$.
The $k$th spatial component of $\bm{\mathcal{B}}_{P}\left(\bm{y}\right)$
is defined as
\[
\bm{\mathcal{B}}_{P,k}\left(\bm{y}\right)=\frac{\rho c^{2}-P}{\rho}\left(\delta_{k1},\ldots,\delta_{kd},0,-W_{1}v_{k},\ldots,-W_{n_{s}}v_{k}\right)^{T},
\]
such that 
\[
\bm{\mathcal{B}}_{P}\left(\bm{y}\right):\nabla\bm{y}=\left(\rho c^{2}-P\right)\nabla\cdot\bm{v}.
\]

\subsection{Discontinuous Galerkin discretization}

To discretize the nonconservative term, we follow the DG scheme presented
in~\citep{Abg23}:
\begin{equation}
\int_{\kappa}\bm{\Phi}\partial_{t}\bm{y}dx+\oint_{\partial\kappa}\bm{\Phi}\bm{\mathcal{F}}^{\dagger}\left(\bm{y}^{+},\bm{y}^{-},\bm{n}\right)ds-\int_{\kappa}\nabla\bm{\Phi}\cdot\bm{\mathcal{F}}\left(\bm{y}\right)dx+\oint_{\partial\kappa}\bm{\Phi}\bm{\mathcal{D}}\left(\bm{y}^{+},\bm{y}^{-},\bm{n}\right)ds+\int_{\kappa}\bm{\Phi}\bm{\mathcal{B}}\left(\bm{y}\right):\nabla\bm{y}=0,\label{eq:DG-discretization-nonconservative}
\end{equation}
where, unless otherwise specified, $\bm{\mathcal{F}}^{\dagger}\left(\bm{y}^{+},\bm{y}^{-},\bm{n}\right)$
is the Lax-Friedrichs numerical flux~(\ref{eq:lax-friedrichs}) and
the last two terms on the LHS approximate the nonconservative flux.
$\bm{\mathcal{D}}\left(\bm{y}^{+},\bm{y}^{-},\bm{n}\right)$ is a
nonconservative jump term for which only the pressure component is
nonzero, defined as
\[
\mathcal{D}_{P}\left(\bm{y}^{+},\bm{y}^{-},\bm{n}\right)=\frac{1}{2}\left.\left(\bm{\mathcal{B}}_{P}\cdot\bm{n}\right)\right|_{\average{\bm{y}}}\cdot\left(\bm{y}^{-}-\bm{y}^{+}\right).
\]
Alternative discretizations of the nonconservative flux are discussed
in~\citep{Lv20}. In~\ref{sec:appendix-compressible-vortex-transport},
we present grid-convergence results for a simple vortex-transport
problem to verify our implementation of the nonconservative terms
in Equation~(\ref{eq:DG-discretization-nonconservative}).

Since the semi-discrete forms of the conservation equations for species
concentrations and momentum remain unchanged, Lemma~\ref{lem:velocity-equilibrium-preservation-standard-dg}
holds. Therefore, the DG scheme~(\ref{eq:DG-discretization-nonconservative})
still preserves velocity equilibrium. 

\subsubsection{Pressure-equilibrium preservation}

\label{subsec:Pressure-equilibrium-preservation-nonconservative}The
pressure-equilibrium-preservation property of~(\ref{eq:DG-discretization-nonconservative})
is analyzed in the following lemma.
\begin{lem}
\label{lem:dg-uncorrected-PEP}In the case of constant pressure $(P=P_{0})$
and velocity $\left(\bm{v}=\bm{v}_{0}\right)$, the DG scheme~(\ref{eq:DG-discretization-nonconservative})
with the Lax-Friedrichs numerical flux~(\ref{eq:lax-friedrichs})
preserves pressure equililbrium (i.e., $\partial_{t}P=0$).
\end{lem}

\begin{proof}
The semi-discrete form of the pressure-evolution equation is
\begin{align*}
\int_{\kappa}\bm{\Phi}\partial_{t}Pdx & =-\oint_{\partial\kappa}\bm{\Phi}P\bm{v}\cdot\bm{n}ds+\int_{\kappa}\nabla\bm{\Phi}\cdot\left(P\bm{v}\right)dx-\oint_{\partial\kappa}\bm{\Phi}\mathcal{D}_{P}\left(\bm{y}^{+},\bm{y}^{-},\bm{n}\right)ds-\int_{\kappa}\bm{\Phi}\mathcal{\bm{B}}_{P}\left(\bm{y}\right):\nabla\bm{y}dx\\
 & =-P_{0}\bm{v}_{0}\cdot\oint_{\partial\kappa}\bm{n}\bm{\Phi}ds+P_{0}\bm{v}_{0}\cdot\int_{\kappa}\nabla\bm{\Phi}dx-\oint_{\partial\kappa}\bm{\Phi}\mathcal{D}_{P}\left(\bm{y}^{+},\bm{y}^{-},\bm{n}\right)ds-\int_{\kappa}\bm{\Phi}\mathcal{\bm{B}}_{P}\left(\bm{y}\right):\nabla\bm{y}dx,
\end{align*}
where the first two terms on the RHS cancel out by a special case
of the divergence theorem. $\mathcal{D}_{P}\left(\bm{y}^{+},\bm{y}^{-},\bm{n}\right)$
is expanded as
\begin{align*}
\mathcal{D}_{P}\left(\bm{y}^{+},\bm{y}^{-},\bm{n}\right) & =\left.\frac{\rho c^{2}-P}{2\rho}\right|_{\average{\bm{y}}}\left(n_{1},\ldots,n_{d},0,-W_{1}v_{0}\cdot n,\ldots,-W_{n_{s}}v\cdot n\right)^{T}\begin{pmatrix}v_{0,1}\left(\rho^{-}-\rho^{+}\right)\\
\vdots\\
v_{0,d}\left(\rho^{-}-\rho^{+}\right)\\
\left(P\right)^{-}-\left(P\right)^{+}\\
\left(C_{1}\right)^{-}-\left(C_{1}\right)^{+}\\
\vdots\\
\left(C_{n_{s}}\right)^{-}-\left(C_{n_{s}}\right)^{+}
\end{pmatrix}\\
 & =\left.\frac{\rho c^{2}-P}{2\rho}\right|_{\average y}\left[\bm{v}_{0}\cdot\bm{n}\left(\rho^{-}-\rho^{+}\right)-\bm{v}_{0}\cdot\bm{n}\left(\rho_{1}^{-}-\rho_{1}^{+}\right)-\ldots-\bm{v}_{0}\cdot\bm{n}\left(\rho_{n_{s}}^{-}-\rho_{n_{s}}^{+}\right)\right]\\
 & =0.
\end{align*}
The final term on the RHS vanishes since $\bm{\mathcal{B}}_{P}\left(\bm{y}\right):\nabla\bm{y}=\left(\rho c^{2}-P\right)\nabla\cdot\bm{v}$
and the velocity is constant. Therefore, $\partial_{t}P=0$.
\end{proof}

\subsection{Correction term: Original formulation}

\label{subsec:correction-term-original}

First, we recast the element-local DG semi-discretization~(\ref{eq:DG-discretization-nonconservative})
as
\begin{equation}
\bm{\mathcal{M}}d_{t}\widehat{\bm{y}}+\widetilde{\bm{\mathcal{R}}}=0,\label{eq:DG-semidiscretization-with-uncorrected-residual}
\end{equation}
where $\widetilde{\mathcal{R}}$ is the uncorrected residual,
\begin{equation}
\widetilde{\bm{\mathcal{R}}}=\oint_{\partial\kappa}\bm{\Phi}\bm{\mathcal{F}}^{\dagger}\left(\bm{y}^{+},\bm{y}^{-},\bm{n}\right)ds-\int_{\kappa}\nabla\bm{\Phi}\cdot\bm{\mathcal{F}}\left(\bm{y}\right)dx+\oint_{\partial\kappa}\bm{\Phi}\bm{\mathcal{D}}\left(\bm{y}^{+},\bm{y}^{-},\bm{n}\right)ds+\int_{\kappa}\bm{\Phi}\bm{\mathcal{B}}\left(\bm{y}\right):\nabla\bm{y}.\label{eq:uncorrected-residual}
\end{equation}
As originally proposed in~\citep{Abg18,Abg22,Abg23}, the DG semi-discretization
is modified as
\begin{equation}
\bm{\mathcal{M}}d_{t}\widehat{\bm{y}}+\bm{\mathcal{R}}=0,\label{eq:DG-semidiscretization-with-corrected-residual}
\end{equation}
where $\bm{\mathcal{R}}$ is a corrected residual, the $k$th component
of which is written as
\[
\bm{\mathcal{R}}_{k}=\widetilde{\bm{\mathcal{R}}}_{k}+\bm{r}_{k},\quad k=1,\ldots,n_{b},
\]
with the correction term, $\bm{r}_{k}$, given by
\begin{equation}
\bm{r}_{k}=\alpha\left(\widehat{\bm{w}}_{k}-\overline{\bm{w}}\right).\label{eq:correction-term-original}
\end{equation}
$\widehat{\bm{w}}$ denotes the coefficients of the projection of
$\bm{w}=\partial_{\bm{y}}\left(\rho e_{t}\right)$, the derivative
of total energy with respect to the state, onto $V_{h}^{p}$, and
$\overline{\bm{w}}$ is defined as
\[
\overline{\bm{w}}=\frac{1}{n_{b}}\sum_{k=1}^{n_{b}}\widehat{\bm{w}}_{k}.
\]
Note that the correction term is conservative since $\sum_{k=1}^{n_{b}}\bm{r}_{k}$
is zero. The projection of $\bm{w}$ onto $V_{h}^{p}$ will be discussed
later. $\bm{w}$ is defined as
\begin{align}
\bm{w} & =\frac{\partial\left(\rho e_{t}\right)}{\partial\bm{y}}=\left(\begin{array}{ccccccc}
\frac{\partial\left(\rho e_{t}\right)}{\partial\rho v_{1}}, & \ldots, & \frac{\partial\left(\rho e_{t}\right)}{\partial\rho v_{d}}, & \frac{\partial\left(\rho e_{t}\right)}{\partial P}, & \frac{\partial\left(\rho e_{t}\right)}{\partial C_{1}}, & \ldots, & \frac{\partial\left(\rho e_{t}\right)}{\partial C_{n_{s}}}\end{array}\right)^{T},\label{eq:total-energy-derivative}
\end{align}
where
\[
\frac{\partial\left(\rho e_{t}\right)}{\partial\rho v_{k}}=v_{k},\quad\frac{\partial\left(\rho e_{t}\right)}{\partial P}=\frac{\rho c_{v}}{R^{0}\sum_{i}C_{i}},\quad\frac{\partial\left(\rho e_{t}\right)}{\partial C_{i}}=W_{i}u_{i}-\frac{\rho c_{v}P}{R^{0}\left(\sum_{j}C_{j}\right)^{2}}-\frac{W_{i}}{2}\bm{v}\cdot\bm{v},
\]
the derivation of which is provided in~\ref{sec:appendix-derivative-total-energy}.
$\alpha$ is a scalar quantity computed from the following constraint
that enforces (semi-discrete) conservation of total energy:
\begin{align}
\sum_{k=1}^{n_{b}}\widehat{\bm{w}}_{k}^{T}\bm{\mathcal{R}}_{k} & =\sum_{k=1}^{N}\left[\widehat{\bm{w}}_{k}^{\cdot T}\widetilde{\bm{\mathcal{R}}}_{k}+\alpha\widehat{\bm{w}}_{k}^{\cdot T}\left(\widehat{\bm{w}}_{k}-\overline{\bm{w}}\right)\right]=\oint_{\partial\kappa}\mathcal{F}_{\rho e_{t}}^{\dagger}\left(\bm{y}^{+},\bm{y}^{-},\bm{n}\right)ds,\label{eq:total-energy-correction-constraint}
\end{align}
where $\mathcal{F}_{\rho e_{t}}^{\dagger}\left(y^{+},y^{-},n\right)$
is a Lax-Friedrichs-type total-energy flux,
\begin{equation}
\mathcal{F}_{\rho e_{t}}^{\dagger}\left(\bm{y}^{+},\bm{y}^{-},\bm{n}\right)=\average{\bm{\mathcal{F}}_{\rho e_{t}}\left(\bm{y}\right)}\cdot\bm{n}+\frac{1}{2}\lambda\left\llbracket \rho e_{t}\right\rrbracket .\label{eq:numerical-flux-energy-lax-friedrichs}
\end{equation}
Defining 
\begin{equation}
\mathcal{E}=\oint_{\partial\kappa}\mathcal{F}_{\rho e_{t}}^{\dagger}\left(\bm{y}^{+},\bm{y}^{-},\bm{n}\right)ds-\sum_{k=1}^{n_{b}}\widehat{\bm{w}}_{k}^{\cdot T}\widetilde{\bm{\mathcal{R}}}_{k}\label{eq:total-energy-consistency-error}
\end{equation}
and solving for $\alpha$ yields
\begin{align*}
\alpha & =\frac{\mathcal{E}}{\sum_{k=1}^{n_{b}}\widehat{\bm{w}}_{k}^{\cdot T}\left(\widehat{\bm{w}}_{k}-\overline{\bm{w}}\right)}\\
 & =\frac{\mathcal{E}}{\sum_{k=1}^{n_{b}}\left(\widehat{\bm{w}}_{k}-\overline{\bm{w}}\right)^{T}\left(\widehat{\bm{w}}_{k}-\overline{\bm{w}}\right)},
\end{align*}
where the second line is due to
\begin{align*}
\sum_{k=1}^{n_{b}}\left(\widehat{\bm{w}}_{k}-\overline{\bm{w}}\right)^{\cdot T}\left(\widehat{\bm{w}}_{k}-\overline{\bm{w}}\right) & =\sum_{k=1}^{n_{b}}\left(\widehat{\bm{w}}_{k}^{T}\widehat{\bm{w}}_{k}-\widehat{\bm{w}}_{k}^{T}\overline{\bm{w}}-\overline{\bm{w}}^{T}\widehat{\bm{w}}_{k}+\overline{\bm{w}}^{T}\overline{\bm{w}}\right)\\
 & =\sum_{k=1}^{n_{b}}\widehat{\bm{w}}_{k}^{\cdot T}\left(\widehat{\bm{w}}_{k}-\overline{\bm{w}}\right)-\overline{\bm{w}}^{T}n_{b}\overline{\bm{w}}+n_{b}\overline{\bm{w}}^{T}\overline{\bm{w}}\\
 & =\sum_{k=1}^{n_{b}}\widehat{\bm{w}}_{k}^{T}\left(\widehat{\bm{w}}_{k}-\overline{\bm{w}}\right),
\end{align*}
such that $\sum_{k=1}^{n_{b}}\widehat{\bm{w}}_{k}^{\cdot T}\left(\widehat{\bm{w}}_{k}-\overline{\bm{w}}\right)\geq0$.
Therefore, the denominator of $\alpha$ is nonnegative and vanishes
only if $\widehat{\bm{w}}_{k}=\overline{\bm{w}},\;k=1,\ldots,n_{b}$
(i.e., the element-local solution is constant), which will be further
discussed later in this section. 

$\widehat{\bm{w}}$ is computed via quadrature-based $L^{2}$ projection.
We adopt similar nomenclature to that in~\citep{Cha18_2}. Let $\left\{ \bm{x}_{q,i}\right\} _{i=1}^{n_{q}}$
and $\left\{ \omega_{q,i}\right\} _{i=1}^{n_{q}}$ denote the points
and weights of a quadrature rule with $n_{q}$ points, where, for
simplicity, it is assumed that the determinant of the geometric Jacobian
is absorbed into the weights. We define $\bm{W}\in\mathbb{R}^{n_{q}\times n_{q}}$
as the diagonal matrix with the quadrature weights along the diagonal
(i.e., $W_{ij}=\omega_{q,i}\delta_{ij}$). We also introduce the quadrature
interpolation matrix,$\bm{V}_{q}\in\mathbb{R}^{n_{q}\times n_{b}}$,
defined as
\[
\left(V_{q}\right)_{ij}=\phi_{j}\left(\bm{x}_{q,i}\right),
\]
as well as the vector of the values of $\bm{w}$ at the quadrature
points, $\bm{w}_{q}=\bm{w}\left(\bm{y}_{q}\right)$, where $\bm{y}_{q}=\bm{V}_{q}\widehat{\bm{y}}$.
$\widehat{\bm{w}}$ is then obtained as
\begin{equation}
\widehat{\bm{w}}=\bm{\mathcal{M}}^{-1}\bm{V}_{q}^{T}\bm{W}\bm{w}_{q}.\label{eq:L2-projection-quadrature}
\end{equation}
Note that with colocated quadrature and a lumped mass matrix, Equation~(\ref{eq:L2-projection-quadrature})
recovers the interpolant of $\bm{w}$. Using this notation and assuming
continuity in time, we have~\citep{Cha18_2}
\begin{equation}
\begin{aligned}\widehat{\bm{w}}^{T}\bm{\mathcal{M}}\frac{d\widehat{\bm{y}}}{dt} & =\bm{w}_{q}^{T}\bm{W}\bm{V}_{q}^{T}\bm{\mathcal{M}}^{-1}\bm{\mathcal{M}}\frac{d\widehat{\bm{y}}}{dt}=\bm{w}_{q}^{T}\bm{W}\bm{V}_{q}^{T}\frac{d\widehat{\bm{y}}}{dt}=\bm{w}_{q}^{T}\bm{W}\frac{d\bm{y}_{q}}{dt}\\
 & =\sum_{i=1}^{n_{q}}\omega_{q,i}\bm{w}\left(\bm{y}\left(\bm{x}_{q,i}\right)\right)\cdot\left.\frac{\partial\bm{y}}{\partial t}\right|_{\bm{x}_{q,i}}=\sum_{i=1}^{n_{q}}\omega_{i}\left.\frac{\partial\rho e_{t}}{\partial t}\right|_{\bm{x}_{q,i}}=\bm{1}^{T}\bm{W}\frac{d\rho e_{t}\left(\bm{y}_{q}\right)}{dt}\\
 & \approx\int_{\kappa}\frac{\partial\rho e_{t}}{\partial t}dx,
\end{aligned}
\label{eq:total-energy-integral-quadrature}
\end{equation}
where, as previously mentioned, $\bm{\mathcal{M}}d_{t}\widehat{\bm{y}}$
is understood to represent the matrix-vector product between the mass
matrix and each of the $m$ components of $d_{t}\widehat{\bm{y}}$
(and similarly for related operations).

In this work, $\widehat{\bm{w}}$ is obtained using quadrature, but
the integrals in Equations~(\ref{eq:uncorrected-residual}) and~(\ref{eq:total-energy-consistency-error})
are computed with a quadrature-free approach~\citep{Atk96,Atk98},
although those integrals can of course be evaluated instead using
quadrature~\citep{Abg18,Abg22,Abg23}. Henceforth, it is implicitly
understood that those integrals (e.g., $\oint_{\partial\kappa}\mathcal{F}_{\rho e_{t}}^{\dagger}\left(\bm{y}^{+},\bm{y}^{-},\bm{n}\right)ds$)
are calculated using numerical integration, whether quadrature-based
or quadrature-free. The reason we explicitly write out the matrices
associated with the above quadrature-based $L^{2}$ projection is
their importance in obtaining a discrete approximation of $\int_{\kappa}\partial_{t}\left(\rho e_{t}\right)dx$
that is subject only to quadrature errors (assuming continuity in
time), which can be controlled via the chosen quadrature rule. Note
that any integration errors in the uncorrected residual~(\ref{eq:uncorrected-residual})
will be absorbed into the correction term. We then have the following
proposition.

\begin{prop}
Assuming continuity in time, the corrected semi-discretization~(\ref{eq:DG-semidiscretization-with-corrected-residual})
satisfies
\begin{equation}
\bm{1}^{T}\bm{W}\frac{d\rho e_{t}\left(\bm{y}_{q}\right)}{dt}+\oint_{\partial\kappa}\mathcal{F}_{\rho e_{t}}^{\dagger}\left(\bm{y}^{+},\bm{y}^{-},\bm{n}\right)ds=0,\label{eq:total-energy-conservation-discrete-local}
\end{equation}
which represents an approximation (subject solely to quadrature error
associated with the first term) of the following statement of local
conservation of total energy:

\[
\int_{\kappa}\frac{\partial\rho e_{t}}{\partial t}dx+\oint_{\partial\kappa}\mathcal{F}_{\rho e_{t}}^{\dagger}\left(\bm{y}^{+},\bm{y}^{-},\bm{n}\right)ds=0.
\]
\end{prop}

\begin{proof}
We take the dot product of $\widehat{\bm{w}}$ with both sides of
the corrected semi-discrete form~(\ref{eq:DG-semidiscretization-with-corrected-residual}).
By Equation~(\ref{eq:total-energy-integral-quadrature}) and since
$\alpha$ is computed such that Equation~(\ref{eq:total-energy-correction-constraint})
is satisfied, we have
\[
\widehat{\bm{w}}^{T}\bm{\mathcal{M}}d_{t}\widehat{\bm{y}}=\bm{1}^{T}\bm{W}\frac{d\rho e_{t}\left(\bm{y}_{q}\right)}{dt}
\]
 and
\[
\sum_{k=1}^{n_{b}}\widehat{\bm{w}}_{k}^{T}\bm{\mathcal{R}}_{k}=\oint_{\partial\kappa}\mathcal{F}_{\rho e_{t}}^{\dagger}\left(\bm{y}^{+},\bm{y}^{-},\bm{n}\right)ds,
\]
which yields Equation~(\ref{eq:total-energy-conservation-discrete-local}).
\end{proof}
\begin{rem}
The correction terms were originally introduced in the context of
steady problems (i.e., no time derivative)~\citep{Abg18}. \citep{Abg22,Abg23}
considered unsteady problems and, at least in the DG context, employed
diagonal mass matrices, where $\widehat{\bm{w}}$ was taken to be
the interpolant of $\bm{w}$ at the nodes. Taking $\widehat{\bm{w}}$
to be the quadrature-based $L^{2}$ projection~(\ref{eq:L2-projection-quadrature})
of $\bm{w}$ thus represents a slight generalization to dense mass
matrices and overintegration.
\end{rem}

\begin{rem}
\label{rem:fully-discrete-conservation-time-integration}In the present
work, we employ standard explicit Runge-Kutta (RK) schemes for time
integration, so fully discrete conservation of total energy is not
achieved. However, the error in total-energy conservation should converge
with the design-order accuracy of the RK scheme. Fully discrete conservation
of total energy can be obtained using implicit methods~\citep{Nor13}
and/or the relaxation approach~\citep{Ran20,Ran20_2}, which can
be applied to explicit RK and multistep schemes. Future work will
explore the importance of fully discrete (as opposed to semidiscrete)
conservation of total energy, especially in shocked flows. Note that
encouraging results for a variety of shock-tube problems were reported
in~\citep{Abg23} with only semidiscrete total-energy conservation,
although small time-step sizes (with explicit RK time integrators)
were potentially used.
\end{rem}

\begin{rem}
\label{rem:correction-term-accuracy}The theoretical order of accuracy
of the correction terms was assessed in~\citep{Abg18}, at least
for steady problems. It was found that there may be slight degradation
of accuracy, although in the specific case of DG schemes where the
correction term is constructed to enforce entropy conservation, design-order
accuracy can be restored. Furthermore, according to Chen and Shu~\citep{Che20},
degeneracy of accuracy may occur in the case of extremely small (relative
to $\mathcal{E}$) correction-term denominator, which is $\mathcal{O}\left(h^{2}\right)$
(where $h$ is the element length scale) due to $\left(\widehat{\bm{w}}_{k}-\overline{\bm{w}}\right)=\mathcal{O}\left(h\right)$,
since it was noted in their error analysis that the numerator and
denominator of the correction term are not necessarily related; see~\citep[Section 5.3]{Che20}
for more details.  Regardless, optimal accuracy has been reported
in various numerical experiments~\citep{Abg18,Abg22,Abg23,Che20},
including those in which the correction term enforces total-energy
conservation~\citep{Abg23} (as opposed to entropy conservation/stability).
Throughout this work, to account for zero or extremely small denominator,
we set $\alpha$ to zero if the denominator is less than $10^{-7}$,
although it is likely that appropriate values of this tolerance are
somewhat dependent on the problem and configuration.
\end{rem}

\subsubsection{Preservation of velocity and pressure equilibria}

In the following proposition, we show that the above DG scheme preserves
neither velocity equilibrium nor pressure equilibrium.
\begin{prop}
\label{prop:PEP-VEP-original-correction}The DG semi-discretization~(\ref{eq:DG-semidiscretization-with-corrected-residual})
with correction term~(\ref{eq:correction-term-original}) is neither
velocity-equilibrium-preserving nor pressure-equilibrium-preserving.
It also does not preserve zero species concentrations.
\end{prop}

\begin{proof}
We know from Lemmas~\ref{lem:velocity-equilibrium-preservation-standard-dg}
and~\ref{lem:dg-uncorrected-PEP} that the uncorrected scheme preserves
velocity and pressure equilibria. Therefore, we only need to analyze
the correction term~(\ref{eq:correction-term-original}). First,
we focus on pressure equilibrium. The correction term can lead to
$\partial_{t}P\neq0$ since $w_{P}=\frac{\partial\left(\rho e_{t}\right)}{\partial P}=\frac{\rho c_{v}}{R^{0}\sum_{i}C_{i}}$
is nonzero even in the case of constant pressure and velocity. From
Equation~(\ref{eq:velocity-evolution-equation}), the contribution
of the correction term to the semi-discrete evolution of velocity
is given by

\begin{equation}
\frac{1}{\rho}\bm{r}_{\rho\bm{v}}-\alpha\frac{\bm{v}_{0}}{\rho}\sum_{i=1}^{n_{s}}W_{i}\bm{r}_{C_{i}}=\alpha\frac{1}{\rho}\left(\widehat{\bm{w}}_{\rho\bm{v}}-\overline{\bm{w}}_{\rho\bm{v}}\right)-\alpha\frac{\bm{v}_{0}}{\rho}\sum_{i=1}^{n_{s}}W_{i}\left(\widehat{\bm{w}}_{C_{i}}-\overline{\bm{w}}_{C_{i}}\right),\label{eq:correction-term-velocity-equilibrium-contribution}
\end{equation}
where $\left(\cdot\right)_{\rho\bm{v}}$ denotes the momentum component
and $\left(\cdot\right)_{C_{i}}$ denotes the component corresponding
to the $i$th species. Given the definition of $\bm{w}$~(\ref{eq:total-energy-derivative}),
the RHS in general does not vanish, which leads to $\partial_{t}\bm{v}\neq0$.
Similarly, since $w_{C_{i}}$ can be nonzero even in the case of zero
$C_{i}$, the correction term does not preserve zero species concentrations.
\end{proof}
\begin{rem}
Since the correction term fails to preserve zero species concentrations,
it can lead to spurious production and destruction of species (including
negative concentrations). Note that zero species concentrations would
not be preserved even if the correction term is constructed to enforce
entropy conservation in a total-energy-based DG scheme~(\ref{eq:dg-mass-matrix-form-conservative}).
However, since the correction term is also conservative, the element-local
average of the species concentration remains zero. As such, bound-preserving
limiters~\citep{Wan12,Zha10,Chi22} can be employed to restore zero
species concentrations. Unfortunately, these limiters, which are fully
conservative in a total-energy-based formulation, are quasi-conservative
in a pressure-based formulation since total energy is a nonlinear
function of the state variables.
\end{rem}

\subsection{Correction term: Modified formulation}

\label{subsec:correction-term-modified}

We now propose simple modifications to the correction terms that,
in the case of pressure and velocity equilibrium, address the issues
discussed in the previous subsection. The first modification is to
change the correction term as
\begin{equation}
\bm{r}_{k}=\alpha\left(\widehat{\bm{z}}_{k}-\overline{\bm{z}}\right),\label{eq:correction-term-modified}
\end{equation}
where $\bm{z}$ is a different set of auxiliary variables. Specifically,
we set
\begin{align}
\bm{z} & =\left(v_{1}\sum_{i=1}^{n_{s}}W_{i}\frac{\partial\left(\rho e_{t}\right)}{\partial C_{i}}\begin{array}{ccccccc}
, & \ldots, & v_{d}\sum_{i=1}^{n_{s}}W_{i}\frac{\partial\left(\rho e_{t}\right)}{\partial C_{i}}, & P, & \frac{\partial\left(\rho e_{t}\right)}{\partial C_{1}}, & \ldots, & \frac{\partial\left(\rho e_{t}\right)}{\partial C_{n_{s}}}\end{array}\right)^{T},\label{eq:correction-variables-new}
\end{align}
where the momentum and pressure components are different from those
of $\bm{w}$. $\alpha$ is now computed as
\begin{align*}
\alpha & =\frac{\mathcal{E}}{\sum_{k=1}^{n_{b}}\widehat{\bm{w}}_{k}^{\cdot T}\left(\widehat{\bm{z}}_{k}-\overline{\bm{z}}\right)}\\
 & =\frac{\mathcal{E}}{\sum_{k=1}^{n_{b}}\left(\widehat{\bm{w}}_{k}-\overline{\bm{w}}\right)^{T}\left(\widehat{\bm{z}}_{k}-\overline{\bm{z}}\right)},
\end{align*}
where the second line is again due to
\begin{align*}
\sum_{k=1}^{n_{b}}\left(\widehat{\bm{w}}_{k}-\overline{\bm{w}}\right)^{\cdot T}\left(\widehat{\bm{z}}_{k}-\overline{\bm{z}}\right) & =\sum_{k=1}^{n_{b}}\left(\widehat{\bm{w}}_{k}^{T}\widehat{\bm{z}}_{k}-\widehat{\bm{w}}_{k}^{T}\overline{\bm{z}}-\overline{\bm{w}}^{T}\widehat{\bm{z}}_{k}+\overline{\bm{w}}^{T}\overline{\bm{z}}\right)\\
 & =\sum_{k=1}^{n_{b}}\widehat{\bm{w}}_{k}^{\cdot T}\left(\widehat{\bm{z}}_{k}-\overline{\bm{z}}\right)-\overline{\bm{w}}^{T}n_{b}\overline{\bm{z}}+n_{b}\overline{\bm{w}}^{T}\overline{\bm{z}}\\
 & =\sum_{k=1}^{n_{b}}\widehat{\bm{w}}_{k}^{T}\left(\widehat{\bm{z}}_{k}-\overline{\bm{z}}\right).
\end{align*}

\begin{rem}
In the case of pressure and velocity equilibria, the denominator is
still guaranteed to be nonnegative. To show this, we define $\Delta\left(\cdot\right)=\widehat{\left(\cdot\right)}-\overline{\left(\cdot\right)}$
and note that $\Delta\bm{w}_{\rho\bm{v}}=0$ since $\bm{w}_{\rho\bm{v}}=\bm{v}.$
In addition, $\Delta\bm{z}_{P}=0$. Therefore,
\[
\Delta\bm{w}^{T}\Delta\bm{z}=\sum_{i=1}^{n_{s}}\left|\Delta\left(\frac{\partial\rho e_{t}}{\partial C_{i}}\right)\right|^{2}\geq0.
\]
$\Delta\bm{w}^{T}\Delta\bm{z}$ is zero if and only if $\widehat{\frac{\partial\rho e_{t}}{\partial C_{i}}}=\overline{\frac{\partial\rho e_{t}}{\partial C_{i}}}$
for all $i$. Note that, in principle, it is possible to have $\widehat{\frac{\partial\rho e_{t}}{\partial C_{i}}}=\overline{\frac{\partial\rho e_{t}}{\partial C_{i}}}$
even if the full state is non-constant, given the definition of $\frac{\partial\rho e_{t}}{\partial C_{i}}$.
However, this seems to be extremely rare. 

At this point, the only difference between the original and modified
specifications of $\alpha$ is in the denominator, where
\[
\Delta\bm{w}^{T}\Delta\bm{w}-\Delta\bm{w}^{T}\Delta\bm{z}=\left|\Delta\left(\frac{\partial\left(\rho e_{t}\right)}{\partial P}\right)\right|^{2}=\left|\Delta\left(\frac{\rho c_{v}}{R^{0}\sum_{i}C_{i}}\right)\right|^{2}.
\]
However, in the absence of pressure and velocity equilibria, $\Delta\bm{w}^{T}\Delta\bm{z}\geq0$
is no longer guaranteed. In this work, since we focus on constant-pressure,
constant-velocity interfacial flows, this is not an issue. A potential
strategy for general flows is to switch between the original and modified
correction terms, where the latter is applied only in regions of the
flow in local pressure and velocity equilibria, which is left for
future work. Note that semi-discrete total-energy conservation would
still be satisfied everywhere.
\end{rem}

The next modification is that, in order to account for elementwise-constant
solutions with inter-element jumps, we incorporate face-based corrections
similar to those presented in~\citep{Abg23}. These face-based corrections,
which are not valid in the case of vanishing inter-element jumps (i.e.,
$\left\llbracket \bm{y}\right\rrbracket =0$), were originally introduced
to construct a scheme without any elementwise corrections; here, we
combine the two types of corrections to account for both elementwise-constant
solutions with inter-element jumps and locally varying solutions with
inter-element continuity (note that the case of elementwise-constant
solutions with inter-element continuity does not need to be addressed
since then there is no temporal change in the state). The face-based
correction is included in the numerical flux as
\begin{equation}
\bm{\mathcal{F}}^{\dagger}\left(\bm{y}^{+},\bm{y}^{-},\bm{n}\right)=\widetilde{\bm{\mathcal{F}}}^{\dagger}\left(\bm{y}^{+},\bm{y}^{-},\bm{n}\right)+\beta\left\llbracket \widehat{\bm{z}}\right\rrbracket ,\label{eq:numerical-flux-modified}
\end{equation}
where $\widetilde{\bm{\mathcal{F}}}^{\dagger}\left(\bm{y}^{+},\bm{y}^{-},\bm{n}\right)$
is the Lax-Friedrichs numerical flux~(\ref{eq:lax-friedrichs}) and
$\beta$ is a coefficient prescribed as
\[
\beta=\frac{-\left\llbracket \bm{\mathcal{F}}_{\rho e_{t}}\left(\bm{y}\right)\right\rrbracket \cdot\bm{n}-\left\llbracket \widehat{\bm{w}}\right\rrbracket ^{T}\widetilde{\bm{\mathcal{F}}}^{\dagger}\left(\bm{y}^{+},\bm{y}^{-},\bm{n}\right)+\left\llbracket \widehat{\bm{w}}^{T}\bm{\mathcal{F}}\left(\bm{y}\right)\right\rrbracket \cdot\bm{n}}{\left\llbracket \widehat{\bm{w}}\right\rrbracket ^{T}\left\llbracket \widehat{\bm{z}}\right\rrbracket }.
\]
The original face-based corrections used $\bm{w}$ in~(\ref{eq:numerical-flux-modified}),
which we replace here with $\bm{z}$, again to preserve velocity and
pressure equilibria. Similar to $\alpha$, in the case of constant
pressure and velocity, the denominator of $\beta$ is nonnegative
and vanishes only if $\left\llbracket \frac{\partial\rho e_{t}}{\partial C_{i}}\right\rrbracket =0$,
for all $i$. The numerical total-energy flux in Equation~(\ref{eq:total-energy-consistency-error})
is also modified as
\begin{equation}
\mathcal{F}_{\rho e_{t}}^{\dagger}\left(\bm{y}^{+},\bm{y}^{-},\bm{n}\right)=\average{\bm{\mathcal{F}}_{\rho e_{t}}\left(\bm{y}\right)}\cdot\bm{n}-\average{\bm{w}^{T}\bm{\mathcal{F}}\left(\bm{y}\right)}\cdot\bm{n}+\average{\bm{w}}^{T}\cdot\bm{\mathcal{F}}^{\dagger}\left(\bm{y}^{+},\bm{y}^{-},\bm{n}\right),\label{eq:numerical-flux-energy-modified}
\end{equation}
which remains conservative and consistent (assuming $\bm{\mathcal{F}}^{\dagger}\left(\bm{y}^{+},\bm{y}^{-},\bm{n}\right)$
is conservative and consistent). The reason for this choice of numerical
total-energy flux will be clarified later in this section. The corrected
numerical flux~(\ref{eq:numerical-flux-modified}) satisfies~\citep{Abg23}
\begin{equation}
\widehat{\bm{w}}^{+}\cdot\left[\bm{\mathcal{F}}^{\dagger}\left(\bm{y}^{+},\bm{y}^{-},\bm{n}\right)-\bm{\mathcal{F}}\left(\bm{y}^{+}\right)\cdot\bm{n}\right]-\widehat{\bm{w}}^{-}\cdot\left[\bm{\mathcal{F}}^{\dagger}\left(\bm{y}^{+},\bm{y}^{-},\bm{n}\right)-\bm{\mathcal{F}}\left(\bm{y}^{-}\right)\cdot\bm{n}\right]=-\left\llbracket \bm{\mathcal{F}}_{\rho e_{t}}\left(\bm{y}\right)\right\rrbracket \cdot\bm{n},\label{eq:numerical-flux-compatibility-condition}
\end{equation}
which will be used below. Note that we only employ the face-based
correction in the case of an elementwise constant solution (i.e.,
the elementwise correction term is not applied).  The rationale for
this is that we have found the face-based correction, in the absence
of additional stabilization, to sometimes detrimentally impact the
solution if $\beta<0$, which has an antidiffusive effect. As such,
the elementwise correction term is preferred, and the face-based correction
is applied only when necessary (though it should be noted that artificial
viscosity, limiting, or other stabilization mechanisms would likely
prevent such undesirable behavior). Furthermore, similar to $\alpha$,
we set $\beta$ to zero if $\left\llbracket \widehat{\bm{w}}\right\rrbracket ^{T}\left\llbracket \widehat{\bm{z}}\right\rrbracket <10^{-6}$
in order to avoid zero or extremely small denominator. 

\subsubsection{Preservation of velocity and pressure equilibria}

We now show that the proposed scheme preserves velocity and pressure
equilibria while maintaining semidiscrete conservation of total energy.
\begin{thm}
\label{thm:PEP-VEP-modified-correction}In the case of constant pressure
($P=P_{0}$) and constant velocity ($\bm{v}=\bm{v}_{0}$), the DG
scheme~(\ref{eq:DG-semidiscretization-with-corrected-residual})
with correction term~(\ref{eq:correction-term-modified}) preserves
velocity and pressure equilibria and (assuming continuity in time)
conserves total energy.
\end{thm}

\begin{proof}
We consider two cases: (a) elementwise-non-uniform solution and (b)
elementwise-constant solution. We begin with the non-uniform case.
Taking the dot product of $\widehat{\bm{w}}$ with both sides of~(\ref{eq:DG-semidiscretization-with-corrected-residual})
(with the correction term~(\ref{eq:correction-term-modified})) gives
\begin{align*}
\widehat{\bm{w}}^{T}\bm{\mathcal{M}}d_{t}\widehat{\bm{y}} & =-\sum_{k=1}^{n_{b}}\widehat{\bm{w}}_{k}^{T}\bm{\mathcal{R}}_{k}=-\sum_{k=1}^{n_{b}}\widehat{\bm{w}}_{k}^{T}\widetilde{\bm{\mathcal{R}}}_{k}-\sum_{k=1}^{n_{b}}\widehat{\bm{w}}_{k}^{T}\bm{r}_{k}\\
 & =-\sum_{k=1}^{n_{b}}\widehat{\bm{w}}^{T}\widetilde{\bm{\mathcal{R}}}-\alpha\sum_{k=1}^{n_{b}}\widehat{\bm{w}}_{k}^{T}\left(\widehat{\bm{z}}_{k}-\overline{\bm{z}}_{k}\right)\\
 & =-\sum_{k=1}^{n_{b}}\widehat{\bm{w}}_{k}^{T}\widetilde{\bm{\mathcal{R}}}_{k}-\frac{\mathcal{E}}{\sum_{k=1}^{n_{b}}\left(\widehat{\bm{w}}_{k}-\overline{\bm{w}}_{k}\right)^{T}\left(\widehat{\bm{z}}_{k}-\overline{\bm{z}}_{k}\right)}\widehat{\bm{w}}_{k}^{T}\left(\widehat{\bm{z}}_{k}-\overline{\bm{z}}_{k}\right)\\
 & =-\sum_{k=1}^{n_{b}}\widehat{\bm{w}}_{k}^{T}\widetilde{\bm{\mathcal{R}}}_{k}-\sum_{k=1}^{n_{b}}\frac{\oint_{\partial\kappa}\mathcal{F}_{\rho e_{t}}^{\dagger}\left(\bm{y}^{+},\bm{y}^{-},\bm{n}\right)ds-\sum_{k=1}^{n_{b}}\widehat{\bm{w}}_{k}^{\cdot T}\widetilde{\bm{\mathcal{R}}}_{k}}{\sum_{k=1}^{n_{b}}\left(\widehat{\bm{w}}_{k}-\overline{\bm{w}}_{k}\right)^{T}\left(\widehat{\bm{z}}_{k}-\overline{\bm{z}}_{k}\right)}\left(\widehat{\bm{w}}_{k}-\overline{\bm{w}}_{k}\right)^{T}\left(\widehat{\bm{z}}_{k}-\overline{\bm{z}}_{k}\right)\\
 & =-\oint_{\partial\kappa}\mathcal{F}_{\rho e_{t}}^{\dagger}\left(\bm{y}^{+},\bm{y}^{-},\bm{n}\right)ds,
\end{align*}
which is a semidiscrete statement of local total-energy conservation.
Due to Lemmas~\ref{lem:velocity-equilibrium-preservation-standard-dg}
and~\ref{lem:dg-uncorrected-PEP}, only the contribution of the correction
term to velocity and pressure equilibria needs to be analyzed. From
Equation~(\ref{eq:velocity-evolution-equation}), the contribution
of the correction term to the semi-discrete evolution of velocity
is given by

\begin{align*}
\frac{\alpha}{\rho}\left[\Delta\bm{z}_{\rho\bm{v}}-\bm{v}_{0}\sum_{i=1}^{n_{s}}W_{i}\Delta\bm{z}_{C_{i}}\right] & =\frac{\alpha}{\rho}\left[\Delta\left(\bm{v}\sum_{i=1}^{n_{s}}W_{i}\frac{\partial\rho e_{t}}{\partial C_{i}}\right)-\bm{v}_{0}\sum_{i=1}^{n_{s}}W_{i}\Delta\left(\frac{\partial\rho e_{t}}{\partial C_{i}}\right)\right]\\
 & =\frac{\alpha}{\rho}\left[\bm{v}_{0}\sum_{i=1}^{n_{s}}W_{i}\Delta\left(\frac{\partial\rho e_{t}}{\partial C_{i}}\right)-\bm{v}_{0}\sum_{i=1}^{n_{s}}W_{i}\Delta\left(\frac{\partial\rho e_{t}}{\partial C_{i}}\right)\right]\\
 & =0,
\end{align*}
which demonstrates velocity-equilibrium preservation. The pressure-equilibrum-preserving
property is obvious given the definition of $\bm{z}$~ (\ref{eq:correction-variables-new}).

Next, we consider the elementwise-constant case (i.e., only the face-based
correction is applied), where we draw from~\citep{Abg23}. Taking
the dot product of $\widehat{\bm{w}}$ with both sides of~(\ref{eq:DG-semidiscretization-with-uncorrected-residual})
(with the corrected numerical flux~(\ref{eq:numerical-flux-modified}))
yields
\begin{align*}
\widehat{\bm{w}}^{T}\bm{\mathcal{M}}d_{t}\widehat{\bm{y}} & =-\sum_{k=1}^{n_{b}}\widehat{\bm{w}}_{k}^{T}\widetilde{\bm{\mathcal{R}}}_{k}\\
 & =-\sum_{k=1}^{n_{b}}\widehat{\bm{w}}_{k}^{T}\oint_{\partial\kappa}\phi_{k}\bm{\mathcal{F}}^{\dagger}\left(\bm{y}^{+},\bm{y}^{-},\bm{n}\right)ds+\sum_{k=1}^{n_{b}}\widehat{\bm{w}}^{T}\int_{\kappa}\nabla\phi_{k}\cdot\bm{\mathcal{F}}\left(\bm{y}\right)dx\\
 & =-\oint_{\partial\kappa}\bm{w}^{T}\bm{\mathcal{F}}^{\dagger}\left(\bm{y}^{+},\bm{y}^{-},\bm{n}\right)ds,
\end{align*}
where the integrals corresponding to the nonconservative flux vanish,
the last line is due to the locally constant state, and (with some
abuse of notation) $\bm{w}=\sum_{k=1}^{n_{b}}\widehat{\bm{w}}_{k}\phi_{k}$.
Adding and subtracting $\oint_{\partial\kappa}\left(\bm{w}^{-}\right)^{T}\bm{\mathcal{F}}^{\dagger}\left(\bm{y}^{+},\bm{y}^{-},\bm{n}\right)ds$
yields
\[
\widehat{\bm{w}}^{T}\bm{\mathcal{M}}d_{t}\widehat{\bm{y}}=-\oint_{\partial\kappa}\frac{1}{2}(\bm{w}^{+}+\bm{w}^{-})^{T}\bm{\mathcal{F}}^{\dagger}\left(\bm{y}^{+},\bm{y}^{-},\bm{n}\right)ds-\oint_{\partial\kappa}\frac{1}{2}(\bm{w}^{+}-\bm{w}^{-})^{T}\bm{\mathcal{F}}^{\dagger}\left(\bm{y}^{+},\bm{y}^{-},\bm{n}\right)ds,
\]
which, upon applying~(\ref{eq:numerical-flux-compatibility-condition}),
can be rewritten as
\[
\widehat{\bm{w}}^{T}\bm{\mathcal{M}}d_{t}\widehat{\bm{y}}=-\frac{1}{2}\oint_{\partial\kappa}(\bm{w}^{+}+\bm{w}^{-})^{T}\bm{\mathcal{F}}^{\dagger}\left(\bm{y}^{+},\bm{y}^{-},\bm{n}\right)ds-\frac{1}{2}\oint_{\partial\kappa}\left\llbracket \bm{w}^{T}\bm{\mathcal{F}}\left(\bm{y}\right)\right\rrbracket \cdot\bm{n}ds+\frac{1}{2}\oint_{\partial\kappa}\left\llbracket \bm{\mathcal{F}}_{\rho e_{t}}\left(\bm{y}\right)\right\rrbracket \cdot\bm{n}ds.
\]
Invoking the identity 
\[
\oint_{\partial\kappa}\bm{1}\cdot\bm{n}ds=\int_{\Omega}\nabla\cdot\bm{1}dx=0,
\]
where $\bm{1}\in\mathbb{R}^{d}$ is a vector of ones, leads to
\begin{align*}
\widehat{\bm{w}}^{T}\bm{\mathcal{M}}d_{t}\widehat{\bm{y}} & =-\oint_{\partial\kappa}\average{\bm{\mathcal{F}}_{\rho e_{t}}\left(\bm{y}\right)}\cdot\bm{n}ds+\oint_{\partial\kappa}\average{\bm{w}^{T}\bm{\mathcal{F}}\left(\bm{y}\right)}\cdot\bm{n}ds-\oint_{\partial\kappa}\average{\bm{w}}^{T}\cdot\bm{\mathcal{F}}^{\dagger}\left(\bm{y}^{+},\bm{y}^{-},\bm{n}\right)ds\\
 & =-\oint_{\partial\kappa}\mathcal{F}_{\rho e_{t}}^{\dagger}\left(\bm{y}^{+},\bm{y}^{-},\bm{n}\right)ds
\end{align*}
since, given the constant state, $\oint_{\partial\kappa}\left(\bm{w}^{+}\right)^{T}\bm{\mathcal{F}}\left(\bm{y}^{+}\right)\cdot\bm{n}ds=0$
and $-\oint_{\partial\kappa}\bm{\mathcal{F}}_{\rho e_{t}}\left(\bm{y}^{+}\right)\cdot\bm{n}ds=0$.
Note that the modified numerical total-energy flux~(\ref{eq:numerical-flux-energy-modified})
is necessary to maintain local total-energy conservation between an
element with a uniform state (i.e., only the face-based correction
is applied) and an element with a non-uniform state (i.e., only the
elementwise correction is applied). From Equation~(\ref{eq:velocity-evolution-equation}),
the contribution of the face-based correction to the semi-discrete
evolution of velocity is given by
\begin{align*}
\frac{\beta}{\rho}\left(\left\llbracket \bm{z}_{\rho\bm{v}}\right\rrbracket -\bm{v}_{0}\sum_{i=1}^{n_{s}}W_{i}\left\llbracket \bm{z}_{C_{i}}\right\rrbracket \right) & =\frac{\beta}{\rho}\left(\left\llbracket \bm{v}\sum_{i=1}^{n_{s}}W_{i}\frac{\partial\rho e_{t}}{\partial C_{i}}\right\rrbracket -\bm{v}_{0}\sum_{i=1}^{n_{s}}W_{i}\left\llbracket \frac{\partial\rho e_{t}}{\partial C_{i}}\right\rrbracket \right)\\
 & =\frac{\beta}{\rho}\left(\bm{v}_{0}\sum_{i=1}^{n_{s}}W_{i}\left\llbracket \frac{\partial\rho e_{t}}{\partial C_{i}}\right\rrbracket -\bm{v}_{0}\sum_{i=1}^{n_{s}}W_{i}\left\llbracket \frac{\partial\rho e_{t}}{\partial C_{i}}\right\rrbracket \right)\\
 & =0.
\end{align*}
The pressure-equilibrum-preserving property is obvious since $\left\llbracket \bm{z}_{P}\right\rrbracket =\left\llbracket P_{0}\right\rrbracket =0.$
\end{proof}
\begin{rem}
\label{rem:correction-term-modified-accuracy-p1}Although $\Delta\bm{w}^{T}\Delta\bm{z}=\sum_{i=1}^{n_{s}}\left[\Delta\left(\frac{\partial\rho e_{t}}{\partial C_{i}}\right)\right]^{2}$
is still $\mathcal{O}\left(h^{2}\right)$ in the case of pressure
and velocity equilibria, the denominator in the modified elementwise
correction term is smaller than that in the original correction term,
which, as discussed in Remark~\ref{rem:correction-term-accuracy},
may lead to greater errors. In our numerical experiments, although
this does not seem to be a significant issue for $p\geq2$, we find
that this can sometimes be problematic for $p=1$, where there are
fewer ``degrees of freedom'' over which to distribute the correction
term to enforce satisfaction of the auxiliary transport equation.
This will be further discussed in Section~\ref{subsec:gaussian-density-wave}.
\end{rem}

\begin{rem}
\label{rem:remove-zero-species-concentration}The modified correction
term~(\ref{eq:correction-term-modified}) still fails to preserve
zero species concentrations. A simple additional modification to guarantee
such preservation is as follows: if $\overline{C}_{l}=0$, then zero
out the component of $\bm{r}_{k}$ corresponding to the $l$th species.
In the case of vanishing inter-element jumps in $C_{l}$, this is
equivalent to removing the $l$th species from the state vector, which
in general can done safely since it otherwise has no effect on the
uncorrected solution. A similar procedure can be performed for the
face-based correction.
\end{rem}

\section{Results}

\label{sec:results}

We apply the proposed formulation to a set of test cases. In the first
one, which involves a Gaussian density wave, we demonstrate optimal
convergence of the formulation. The next two consist of high-velocity
and low-velocity advection of a nitrogen/n-dodecane thermal bubble
in one dimension, where we assess the ability of the methodology to
preserve pressure equilibrium and conserve total energy. The final
two configurations are two- and three-dimensional versions of the
high-velocity thermal-bubble test case. Curved elements are also considered.
The following schemes are employed in this section:
\begin{itemize}
\item P1: pressure-based DG scheme without any correction term
\item P2: pressure-based DG scheme with the original correction term~(\ref{eq:correction-term-original})
(Section~\ref{subsec:correction-term-original})
\item P3: pressure-based DG scheme with modified corrections~(\ref{eq:correction-term-modified})
and~(\ref{eq:numerical-flux-modified}) (Section~\ref{subsec:correction-term-modified})
\item E1: total-energy-based DG scheme with overintegration
\item E2: total-energy-based DG scheme with colocated integration
\end{itemize}
The pressure-based solutions are computed with overintegration. All
solutions are initialized using interpolation and advanced in time
using the explicit, third-order, strong-stability-preserving Runge-Kutta
(SSPRK3) scheme~\citep{Got01}. The CFL is defined here as
\begin{equation}
\mathrm{CFL}=\frac{\Delta t\left(2p+1\right)}{h}\left(\left|\bm{v}\right|+c\right),\label{eq:cfl}
\end{equation}
where $h$ is an element length scale. No artificial viscosity or
limiting is applied. The governing equations are discretized in nondimensional
form~\citep{Joh22}. All simulations are performed using a modified
version of the JENRE\textregistered~Multiphysics Framework~\citep{Joh20_2}
that incorporates the extensions described in this work.

\subsection{Gaussian density wave}

\label{subsec:gaussian-density-wave}

To investigate the order of accuracy of the scheme, we consider the
advection of a multicomponent Gaussian wave based on the test in~\citep{Tro23}.
A periodic domain $\Omega=[-0.5,0.5]$ is initialized in nondimensional
form as follows:
\begin{eqnarray}
v & = & 5,\nonumber \\
Y_{1} & = & \frac{1}{2}\left[\sin\left(2\pi x\right)+1\right],\nonumber \\
Y_{2} & = & 1-Y_{1},\label{eq:Gaussian-wave}\\
\rho & = & \exp\left(-\sigma x^{2}\right)+4,\nonumber \\
P & = & 2,\nonumber 
\end{eqnarray}
where $\sigma=500$. The thermodynamic relations for the two fictitious
species considered here are given by
\begin{align*}
\frac{W_{1}c_{p,1}\left(T\right)}{R^{0}} & =3.5,\quad\frac{W_{1}h_{1}\left(T\right)}{R^{0}}=3.5T,\\
\frac{W_{2}c_{p,2}\left(T\right)}{R^{0}} & =2.491,\quad\frac{W_{2}h_{2}\left(T\right)}{R^{0}}=2.491T.
\end{align*}
Four element sizes are considered: $h$, $h/2$, $h/4$, and $h/8$,
where $h=0.02$. The solution is integrated in time using a CFL of
$0.1$ in order to ensure sufficiently low temporal errors. The $L^{2}$
error at $t=0.2$, corresponding to one advection period, is calculated
in terms of the following normalized state variables: 
\[
\underline{\rho v}_{k}=\frac{1}{\sqrt{\rho_{r}P_{r}}}\rho v_{k},\quad\underline{\zeta}=\frac{1}{P_{r}}\zeta,\quad\underline{C}_{i}=\frac{R^{0}T_{r}}{P_{r}}C_{i},
\]
where $\zeta$ is either total energy or pressure (depending on the
formulation) and $\rho_{r}=1\,\mathrm{kg\cdot}\mathrm{m}^{-3}$ and
$P_{r}=101325\,\mathrm{Pa}$ are reference values. Figure~\ref{fig:low_density_sinusoidal_convergence.}
displays the results of the convergence study from $p=1$ to $p=3$
for the corrected DG scheme with the original correction term~(\ref{eq:correction-term-original})
and the modified correction term~(\ref{eq:correction-term-modified}).
The dashed lines represent the optimal convergence rates of $p+1$.
As discussed in Remark~\ref{rem:correction-term-modified-accuracy-p1},
the smaller denominator in the modified correction term~(\ref{eq:correction-term-modified})
may lead to larger errors, which seem to manifest most dramatically
in the $p=1$ case. Here, these errors accumulate and lead to solver
divergence on the coarsest mesh, so Figure~\ref{fig:gaussian-convergence-modified}
shows results only for $h/2$, $h/4$, and $h/8$ for $p=1$. In general,
optimal order of accuracy is observed in Figure~\ref{fig:low_density_sinusoidal_convergence.}.

\begin{figure}[tbph]
\begin{centering}
\subfloat[\label{fig:gaussian-convergence-uncorrected}No correction term.]{\includegraphics[width=0.48\columnwidth]{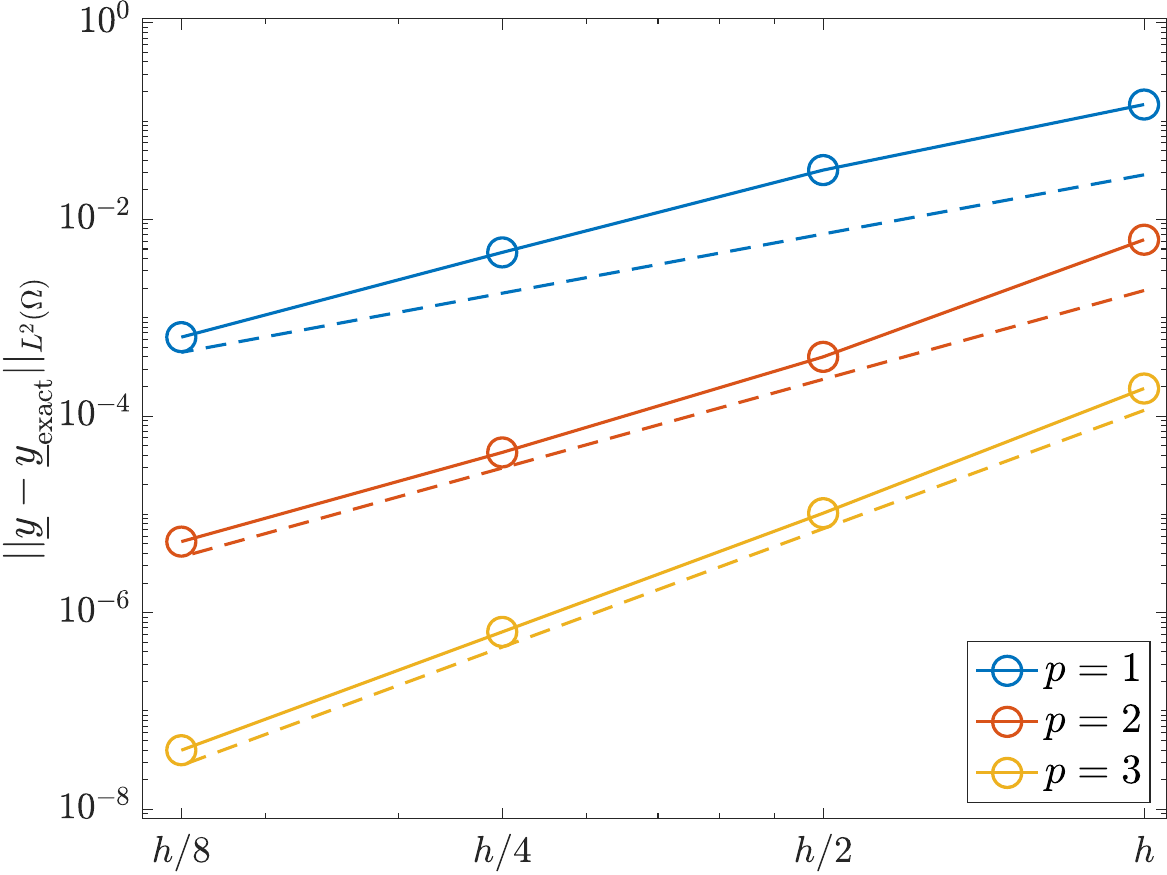}}\hfill{}\subfloat[\label{fig:gaussian-convergence-original}Original correction term~(\ref{eq:correction-term-original}).]{\includegraphics[width=0.48\columnwidth]{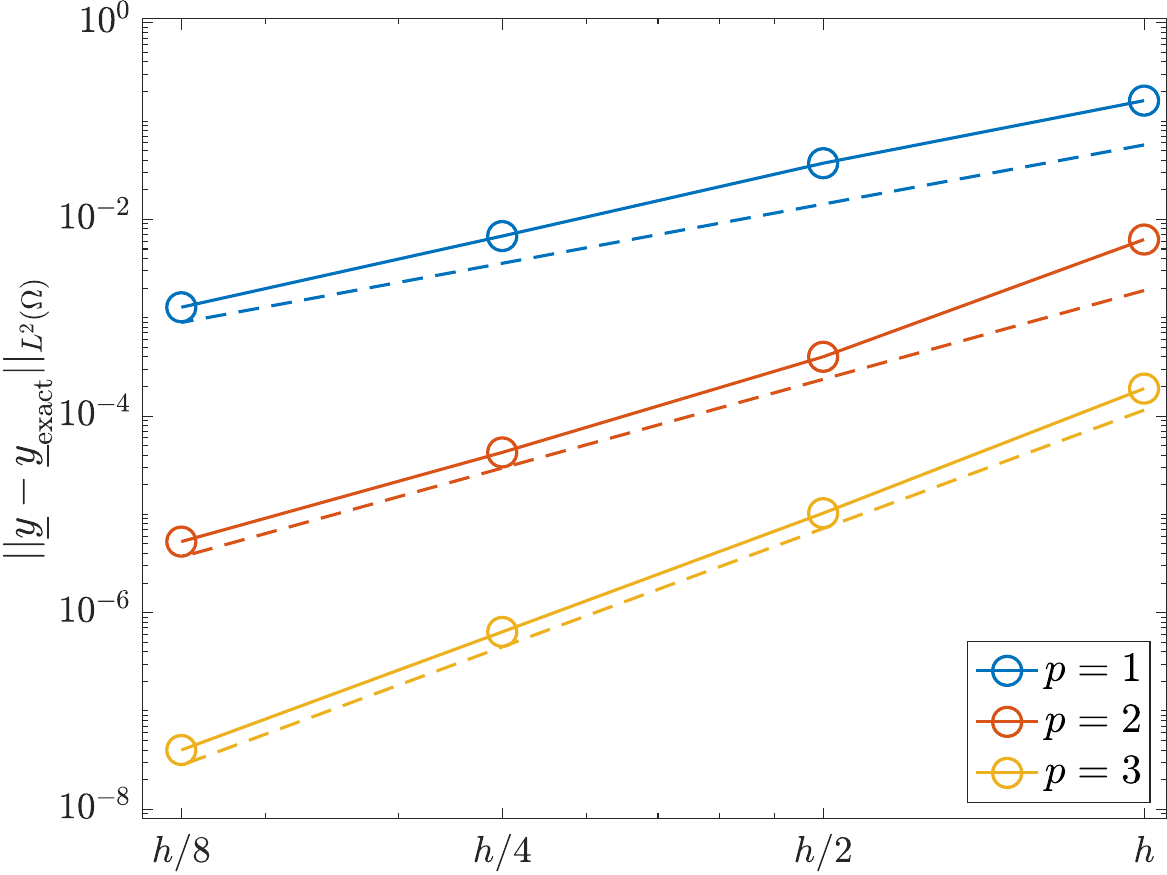}}\hfill{}\subfloat[\label{fig:gaussian-convergence-modified}Modified correction term~(\ref{eq:correction-term-modified}).]{\includegraphics[width=0.48\columnwidth]{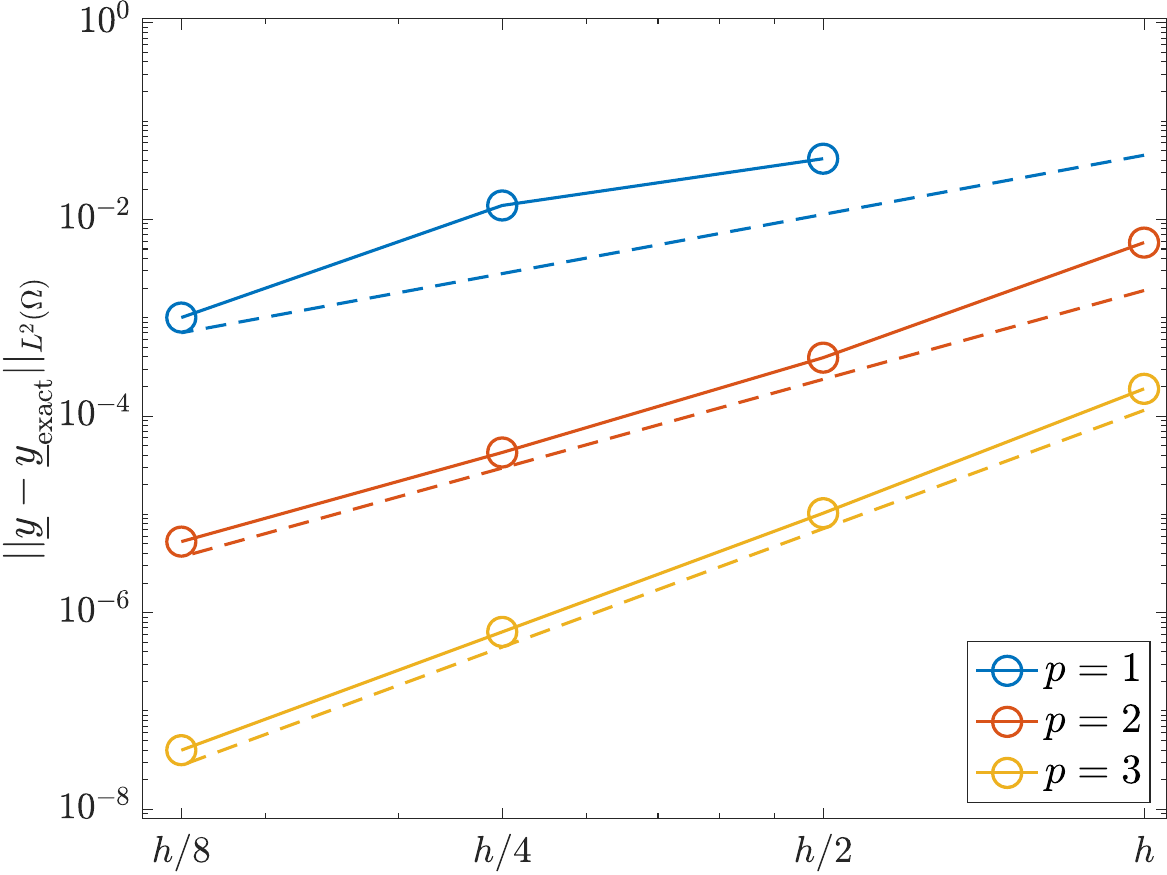}}
\par\end{centering}
\centering{}\caption{\label{fig:low_density_sinusoidal_convergence.}Convergence under
grid refinement, with $h=0.02$, for the multicomponent Gaussian-wave
test. The $L^{2}$ error of the normalized state with respect to the
exact solution at $t=1$ is computed. The dashed lines represent convergence
rates of $p+1$.}
\end{figure}

\subsection{One-dimensional, high-velocity thermal bubble}

\label{subsec:thermal-bubble-600-m-s-1D}

This section presents results for the high-velocity advection of a
high-pressure, nitrogen/n-dodecane thermal bubble. This test case
was originally presented in~\citep{Ma17} and~\citep{Boy21} with
discontinuous initial conditions and then modified to be smooth in~\citep{Chi23_short}
in order to circumvent the need for discontinuity-capturing techniques.
The initial condition is given by
\begin{eqnarray}
v & = & 600\textrm{ m/s},\nonumber \\
Y_{n\text{-}\mathrm{C_{12}H_{26}}} & = & \frac{1}{2}\left[1-\tanh\left(25|x|-5\right)\right],\nonumber \\
Y_{\mathrm{N_{2}}} & = & 1-Y_{n\text{-}\mathrm{C_{12}H_{26}}},\label{eq:thermal-bubble}\\
T & = & \frac{T_{\min}+T_{\max}}{2}+\frac{T_{\max}-T_{\min}}{2}\tanh\left(25|x|-5\right)\textrm{ K},\nonumber \\
P & = & 6\textrm{ MPa},\nonumber 
\end{eqnarray}
where $T_{\min}=363\text{ K}$ and $T_{\max}=900\text{ K}$. The computational
domain is $\left[-0.5,0.5\right]\text{ m}$, partitioned into $25$
line cells (such that $h=0.04\;\mathrm{m}$), and $p=3$ solutions
are obtained. Both sides of the domain are periodic. Note that the
assumption of thermally perfect gases may not be physically accurate
at these high-pressure conditions; a real-fluid model (e.g., a cubic
equation of state~\citep{Ma17,Boy21}) would be more appropriate.
Nevertheless, this study concerns numerical accuracy, robustness,
and pressure-equilibrium preservation, which we have found this case
to be well-suited to assess~\citep{Chi23_short}. 

As discussed in Remark~\ref{rem:fully-discrete-conservation-time-integration},
the use of explicit SSPRK schemes prevents fully discrete conservation
of total energy. However, the energy-conservation error should converge
at the expected rate associated with the time-integration scheme.
We perform a temporal convergence study for the pressure-based schemes
(P1, P2, and P3). The percent error in total-energy conservation is
computed after 100 advection periods (i.e., $t=100\tau$, where $\tau$
is the time to advect the solution one period) as
\[
\left|\frac{\int_{\Omega}\left.\rho e_{t}\right|_{t=100\tau}dx-\int_{\Omega}\left.\rho e_{t}\right|_{t=0}dx}{\int_{\Omega}\left.\rho e_{t}\right|_{t=0}dx}\right|\times100.
\]
Five time-step sizes are considered: $\Delta t$, $\Delta t/2$, $\Delta t/4$,
$\Delta t/8$, and $\Delta t/16$, where $\Delta t=3.14\text{ \ensuremath{\mu s}}$.
The results are given in Figure~\ref{fig:temporal_convergence},
where the dashed line represents the expected third-order convergence
rate. Without the correction term (P1), the energy-conservation error
does not converge with decreasing time-step size. The expected third-order
convergence rate is observed with the correction terms (P2 and P3)
combined with SSPRK3 time integration. 

\begin{figure}[tbph]
\begin{centering}
\includegraphics[width=0.48\columnwidth]{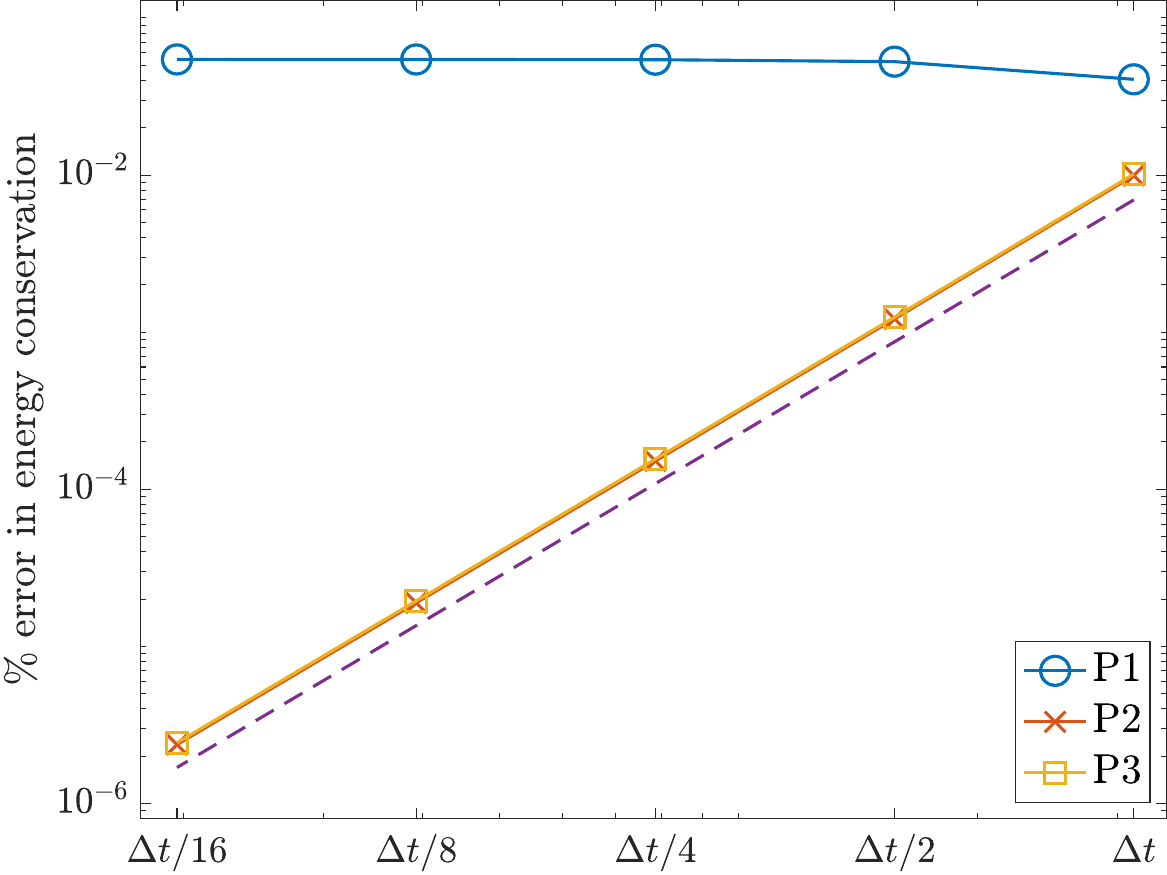}
\par\end{centering}
\caption{\label{fig:temporal_convergence}Convergence of total-energy conservation
with respect to time-step size, where $\Delta t=3.14\text{ \ensuremath{\mu s}}$,
for the one-dimensional, high-velocity thermal-bubble test case. The
percent error in global energy conservation after 100 advection periods
is computed. The dashed line represents a third-order rate of convergence.}
\end{figure}

Next, we assess the robustness and pressure-equilibrium errors of
all considered schemes. All solutions are integrated in time until
$t=100\tau$ (unless the solution diverges) with $\mathrm{CFL}=0.6$.
For comparison, we also employ the total-energy-based DG scheme~(\ref{eq:dg-mass-matrix-form-conservative})
with overintegration (E1) and colocated integration (E2). Figure~\ref{fig:thermal_bubble_dodecane_deltaP_600_m-s}
presents the temporal variation of the percent error in pressure.
The percent error in pressure is sampled every $\tau$ seconds until
the solution either diverges or is advected for 100 periods. The total-energy-based
DG scheme with colocated integration diverges rapidly. With overintegration,
the total-energy-based DG scheme remains stable but exhibits large
errors in pressure. In contrast, the pressure-based DG scheme with
the original correction term~(\ref{eq:correction-term-original})
(P2) yields significantly smaller (though not negligible) errors in
pressure. The pressure deviations associated with the uncorrected
pressure-based DG scheme (P1) and the proposed pressure-based DG scheme
(P3) are due to finite-precision-induced errors.

Figures~\ref{fig:thermal_bubble_dodecane_P_600_m-s} and~\ref{fig:thermal_bubble_dodecane_T_600_m-s}
display the pressure and temperature profiles, respectively, at $t=100\tau$
for the stable solutions. Large-scale oscillations are observed in
the overintegrated total-energy-based solution. These oscillations
are significantly reduced or eliminated in the pressure-based solutions,
which are all visually similar.

\begin{figure}[H]
\begin{centering}
\subfloat[\label{fig:thermal_bubble_dodecane_deltaP_600_m-s}Temporal variation
of error in pressure.]{\includegraphics[width=0.48\columnwidth]{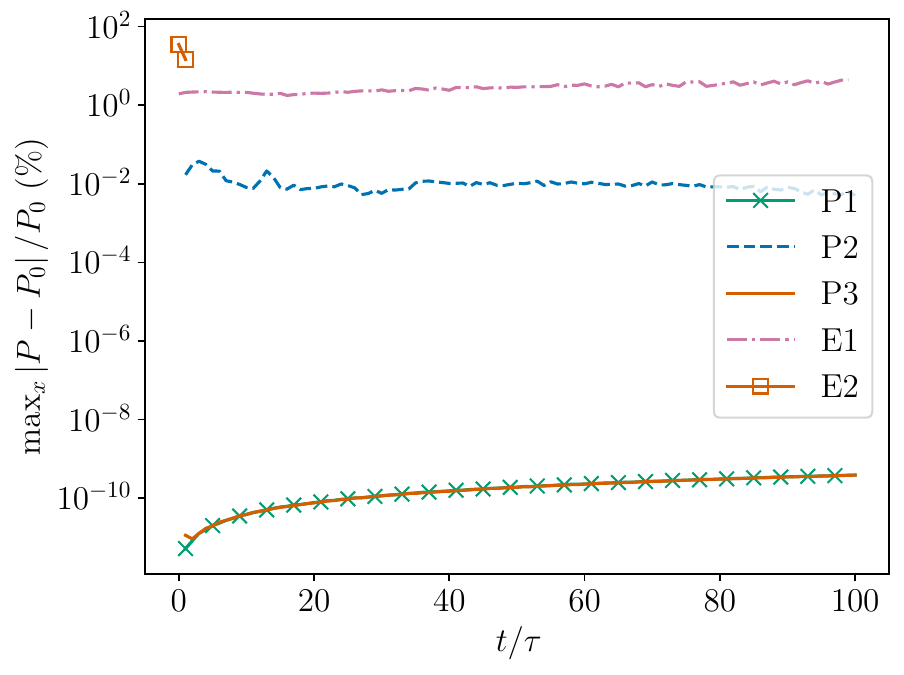}}\hfill{}\subfloat[\label{fig:thermal_bubble_dodecane_P_600_m-s}Pressure profiles at
$t=100\tau$.]{\includegraphics[width=0.48\columnwidth]{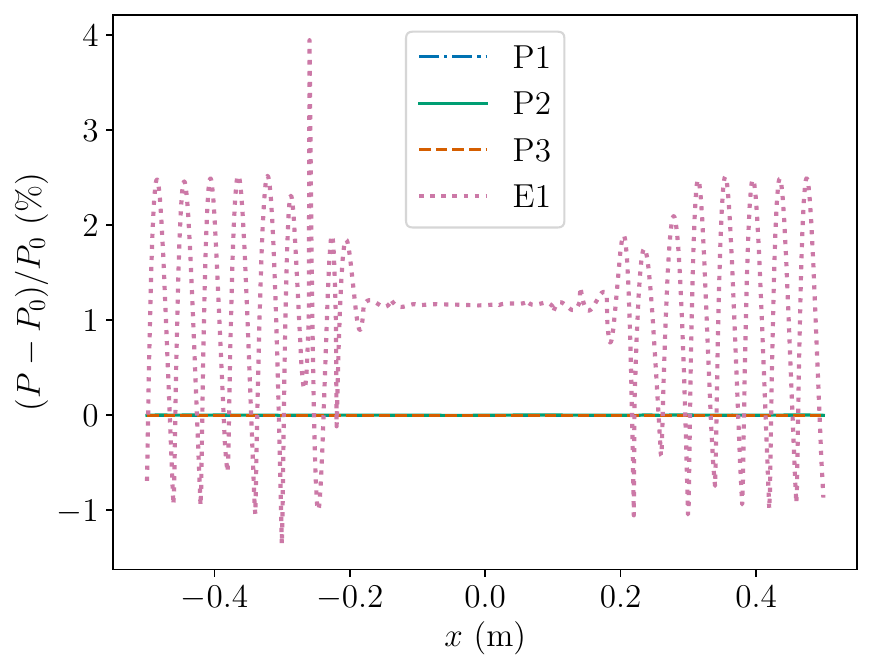}}\hfill{}\subfloat[\label{fig:thermal_bubble_dodecane_T_600_m-s}Temperature profiles
at $t=100\tau$.]{\includegraphics[width=0.48\columnwidth]{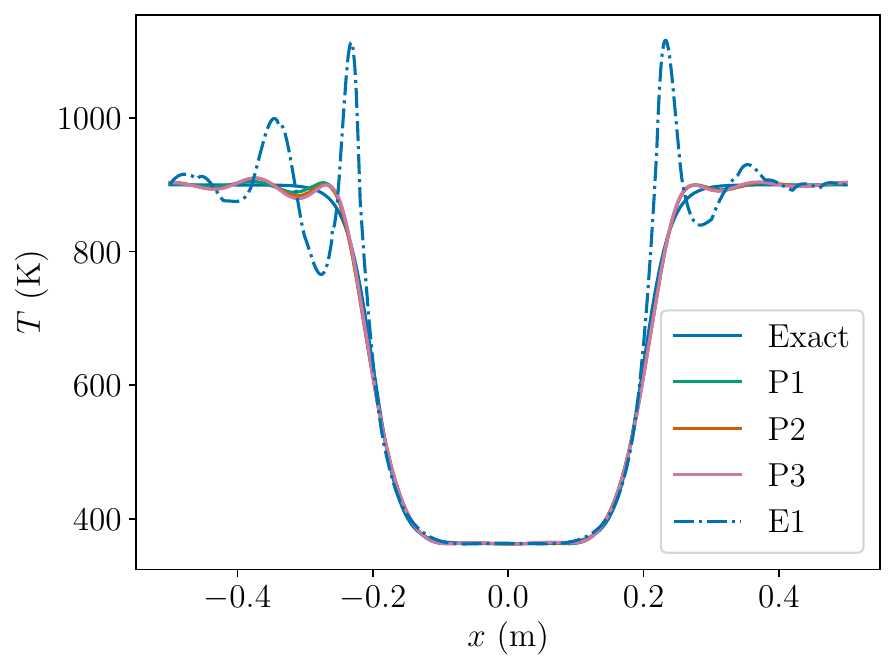}}
\par\end{centering}
\centering{}\caption{\label{fig:thermal_bubble_dodecane_600_m-s}$p=3$ solutions to the
advection of a nitrogen/n-dodecane thermal bubble at $v=600\text{ m/s}$.
P1: uncorrected pressure--based DG scheme; P2: pressure-based DG
scheme with original correction term~(\ref{eq:correction-term-original});
P3: proposed pressure-based DG scheme (Section~\ref{subsec:correction-term-modified});
E1: total-energy-based DG scheme with overintegration; E2: total-energy-based
DG scheme with colocated integration; Exact: exact solution.}
\end{figure}

Next, to demonstrate the failure to preserve zero species concentrations
of the original correction term, we include another species, O\textsubscript{2},
and set $Y_{\mathrm{O_{2}}}=0$ at $t=0$. Of course, $Y_{\mathrm{O_{2}}}$
should remain zero for the entire simulation. However, as shown in
Figure~\ref{fig:thermal_bubble_dodecane_600_m-s_Y_O2}, there is
spurious production and destruction of O\textsubscript{2} with the
P2 scheme. This can be particularly detrimental in the case of chemically
reacting flows since even trace concentrations of highly reactive
species can significantly affect the solution~\citep{Mag98,Cra19,Chi24_ozone}.
With the strategy discussed in Remark~\ref{rem:remove-zero-species-concentration},
the proposed (P3) scheme preserves zero $Y_{\mathrm{O_{2}}}$.

\begin{figure}[tbph]
\begin{centering}
\includegraphics[width=0.48\columnwidth]{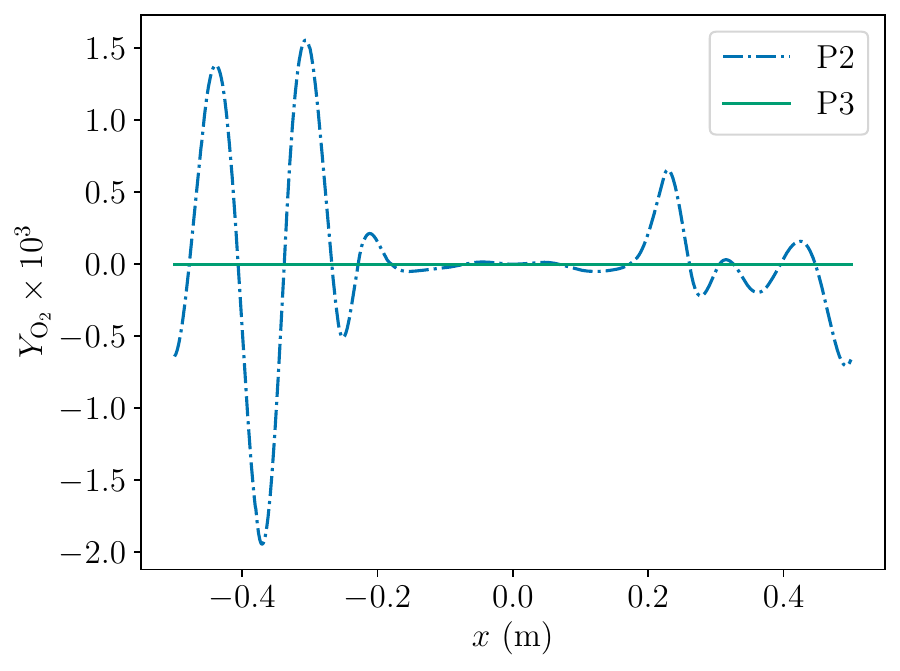}
\par\end{centering}
\caption{\label{fig:thermal_bubble_dodecane_600_m-s_Y_O2}$Y_{\mathrm{O_{2}}}$
profiles at $t=100\tau$ obtained from $p=3$ solutions to the advection
of a nitrogen/n-dodecane thermal bubble at $v=600\text{ m/s}$.}
\end{figure}

\subsection{One-dimensional, low-velocity thermal bubble}

\label{subsec:thermal-bubble-1-m-s-1D}

The initial condition here is identical to that for the previous test
case, except the velocity is set to $v=1\text{ m/s}$. In~\citep{Chi23_short},
noticeably different results (in terms of stability) were observed
between the low-velocity and high-velocity case when employing a total-energy-based
DG scheme with various integration strategies. The domain is partitioned
into 50 cells, and $p=2$ solutions are integrated in time for ten
advection periods (unless the solver crashes) with a CFL of 0.8. 

Figure~\ref{fig:thermal_bubble_dodecane_deltaP_1_m-s} presents the
temporal variation of the percent error in pressure for all considered
schemes, which is sampled every $\tau/10$ seconds (where $\tau$
is again the time to advect the solution one period) until the solution
either diverges or is advected for ten periods. Unlike in the high-velocity
case, the total-energy-based DG scheme with overintegration diverges
rapidly. With colocated integration, the total-energy-based DG scheme
remains stable but exhibits appreciable errors in pressure. Conversely,
all pressure-based DG schemes exhibit significantly smaller errors
in pressure. These errors are still present with the original correction
term~(\ref{eq:correction-term-original}) and completely eliminated
(apart from errors induced by finite precision) with the proposed
modifications. Figures~\ref{fig:thermal_bubble_dodecane_P_1_m-s}
and~\ref{fig:thermal_bubble_dodecane_T_1_m-s} display the pressure
and temperature profiles, respectively, at $t=10\tau$ for the stable
solutions. The colocated total-energy-based solution exhibits a nearly
uniform offset from the equilibrium pressure and large-scale temperature
oscillations. Significantly better agreement with the exact solution
is observed for the pressure-based solutions.

\begin{figure}[H]
\begin{centering}
\subfloat[\label{fig:thermal_bubble_dodecane_deltaP_1_m-s}Temporal variation
of error in pressure.]{\includegraphics[width=0.48\columnwidth]{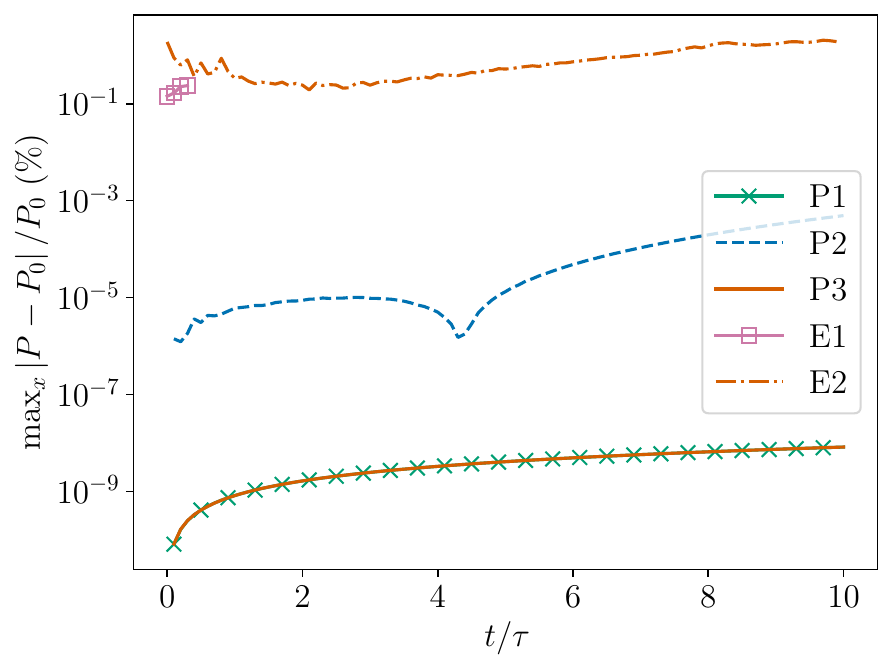}}\hfill{}\subfloat[\label{fig:thermal_bubble_dodecane_P_1_m-s}Pressure profiles at $t=10\tau$.]{\includegraphics[width=0.48\columnwidth]{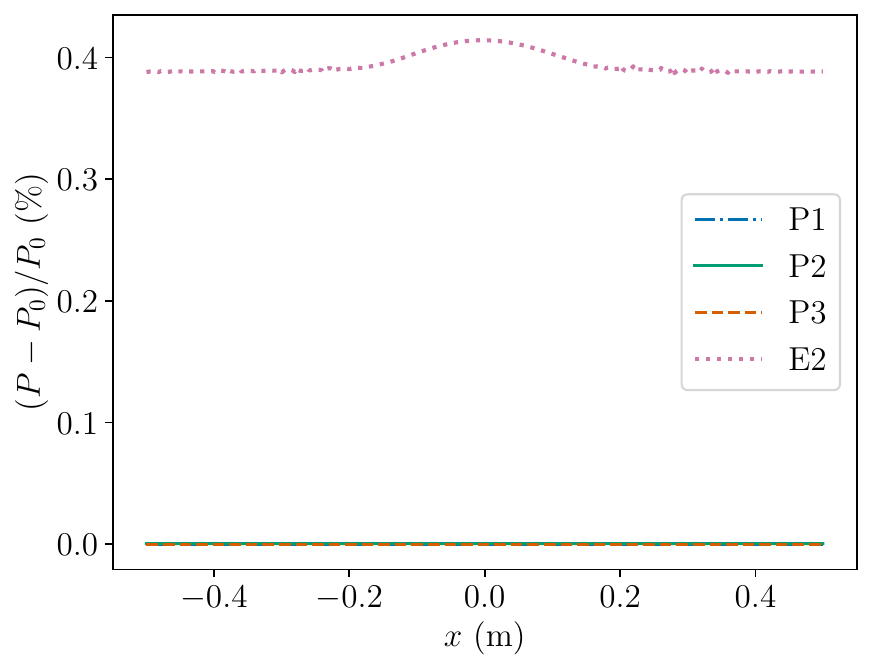}}\hfill{}\subfloat[\label{fig:thermal_bubble_dodecane_T_1_m-s}Temperature profiles at
$t=10\tau$.]{\includegraphics[width=0.48\columnwidth]{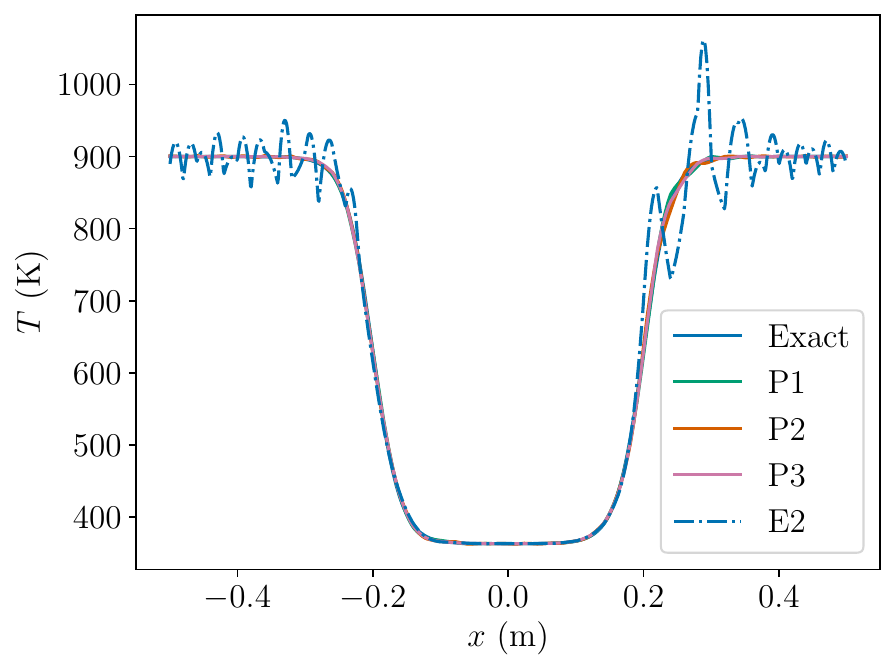}}
\par\end{centering}
\centering{}\caption{\label{fig:thermal_bubble_dodecane_1_m-s}$p=3$ solution to the advection
of a nitrogen/n-dodecane thermal bubble at $v=1\text{ m/s}$. P1:
uncorrected pressure--based DG scheme; P2: pressure-based DG scheme
with original correction term~(\ref{eq:correction-term-original});
P3: proposed pressure-based DG scheme (Section~\ref{subsec:correction-term-modified});
E1: total-energy-based DG scheme with overintegration; E2: total-energy-based
DG scheme with colocated integration; Exact: exact solution.}
\end{figure}

\subsection{Two-dimensional, high-velocity thermal bubble}

\label{subsec:thermal-bubble-600-m-s-2D}

This problem is a two-dimensional version of the high-velocity thermal-bubble
test case in Section~\ref{subsec:thermal-bubble-600-m-s-1D}. The
initial condition is given by
\begin{eqnarray}
Y_{n\text{-}\mathrm{C_{12}H_{26}}} & = & \frac{1}{2}\left[1-\tanh\left(25\sqrt{x_{1}^{2}+x_{2}^{2}}-5\right)\right],\nonumber \\
Y_{\mathrm{N_{2}}} & = & 1-Y_{n\text{-}\mathrm{C_{12}H_{26}}},\nonumber \\
T & = & \frac{T_{\min}+T_{\max}}{2}+\frac{T_{\max}-T_{\min}}{2}\tanh\left(25\sqrt{x_{1}^{2}+x_{2}^{2}}-5\right)\textrm{ K},\label{eq:thermal-bubble-2d}\\
P & = & 6\textrm{ MPa},\nonumber \\
\left(v_{1},v_{2}\right) & = & \left(600,0\right)\textrm{ m/s},\nonumber 
\end{eqnarray}
The computational domain is $\text{\ensuremath{\Omega}}=\left[-0.5,0.5\right]\mathrm{m}\times\left[-0.5,0.5\right]\mathrm{m}$.
The left, right, bottom, and top boundaries are periodic. Gmsh~\citep{Geu09}
is used to generate an unstructured triangular mesh with a characteristic
cell size of $h=0.04~\mathrm{m}$. Figure~\ref{fig:thermal_bubble_2d_initial}
presents the initial $p=3$ density and temperature fields, superimposed
by the mesh.
\begin{figure}[h]
\subfloat[\label{fig:thermal_bubble_2d_initial_density}Initial density field.]{\includegraphics[width=0.48\columnwidth]{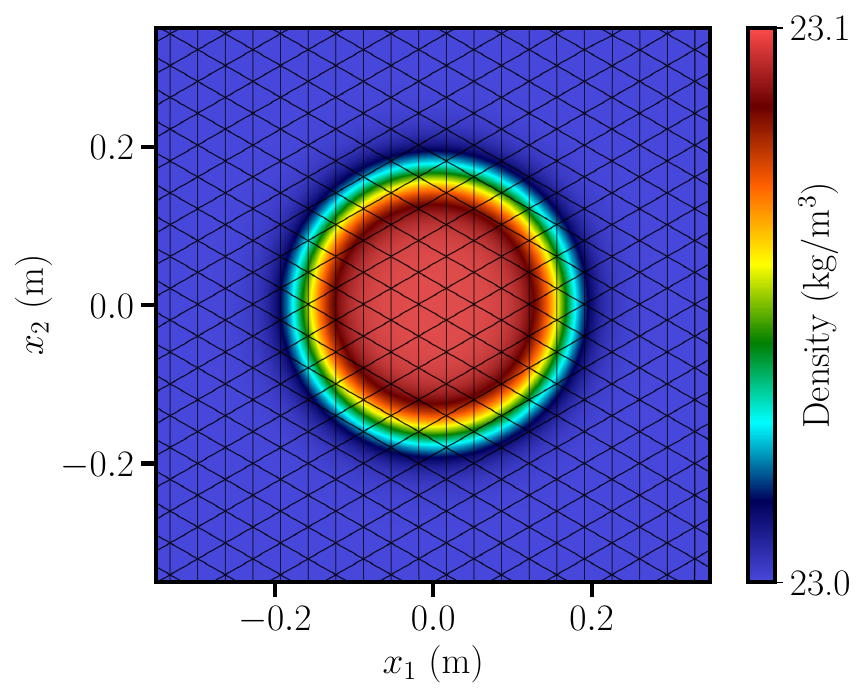}}\hfill{}\subfloat[\label{fig:thermal_bubble_2d_initial_temperature}Initial temperature
field.]{\includegraphics[width=0.48\columnwidth]{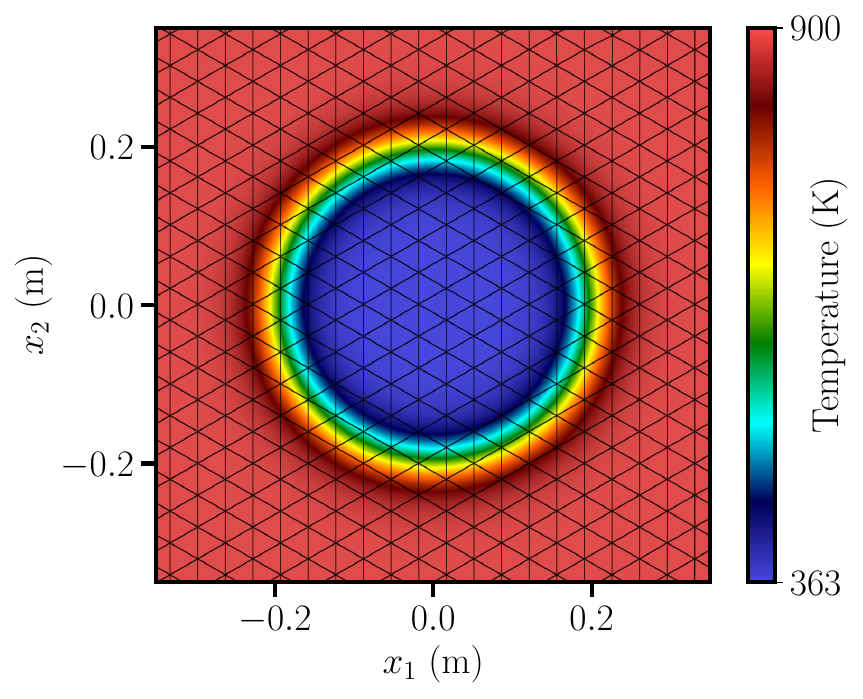}}

\caption{\label{fig:thermal_bubble_2d_initial}Initial $p=3$ density and pressure
fields, superimposed by the grid, for the two-dimensional advection
of a nitrogen/n-dodecane thermal bubble.}
\end{figure}

We first assess the temporal convergence of the pressure-based schemes
in this two-dimensional setting. Four time-step sizes are considered:
$\Delta t/2$, $\Delta t/4$, $\Delta t/8$, and $\Delta t/16$, where
$\Delta t=3.14\text{ \ensuremath{\mu s}}$, as in Section~\ref{subsec:thermal-bubble-600-m-s-1D}.
Note that, unlike in the one-dimensional version, $\Delta t$ is too
large and results in solver divergence. The results are presented
in Figure~\ref{fig:temporal_convergence}, where the dashed line
represents the expected third-order convergence rate. The P2 solution
diverges before $t=100\tau$ and is therefore not included. The expected
third-order convergence rate is observed with the P3 scheme combined
with SSPRK3 time integration. The P1 scheme (no correction term) results
in significantly greater errors in energy conservation, which fail
to converge with decreasing time-step size.
\begin{figure}[tbph]
\begin{centering}
\includegraphics[width=0.48\columnwidth]{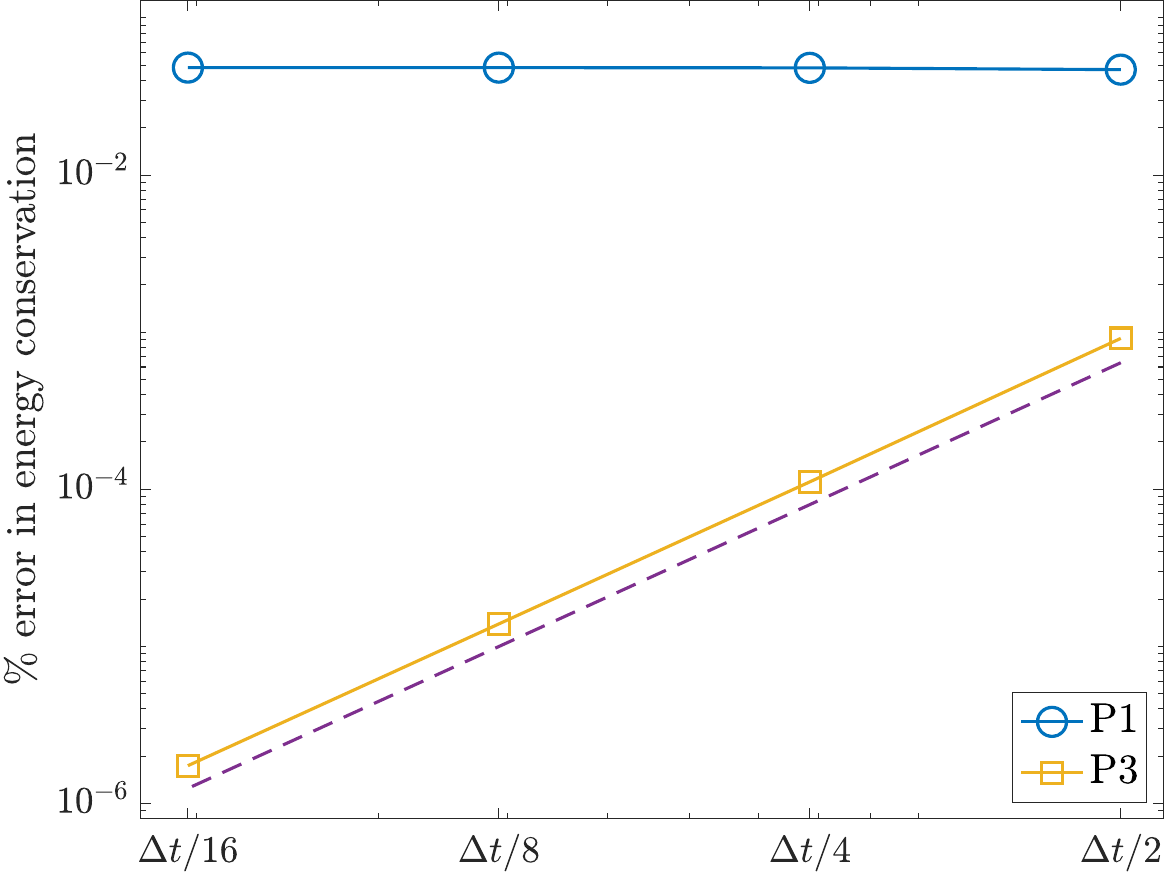}
\par\end{centering}
\caption{\label{fig:temporal_convergence-2d}Convergence of total-energy conservation
with respect to time-step size, where $\Delta t=3.14\text{ \ensuremath{\mu s}}$,
for the two-dimensional, high-velocity thermal-bubble test case. The
percent error in global energy conservation after 100 advection periods
is computed. The dashed line represents a third-order rate of convergence.}
\end{figure}

Figure~\ref{fig:thermal_bubble_dodecane_deltaP_600_m-s_2D} presents
the temporal variation of percent error in pressure, sampled every
$\tau$ seconds, for all considered schemes. Specifically, $p=3$
solutions are integrated in time with a CFL of 0.6 until either $t=100\tau$
or the solution diverges. The total-energy-based DG scheme with colocated
integration diverges rapidly. Unlike in the one-dimensional setting,
the total-energy-based DG scheme with colocated integration and the
pressure-based DG scheme with the original correction term~(\ref{eq:correction-term-original})
diverge before $t=100\tau$. The small pressure deviations observed
for the uncorrected pressure-based DG scheme and the proposed scheme
are due to finite-precision issues.
\begin{figure}[tbph]
\begin{centering}
\includegraphics[width=0.48\columnwidth]{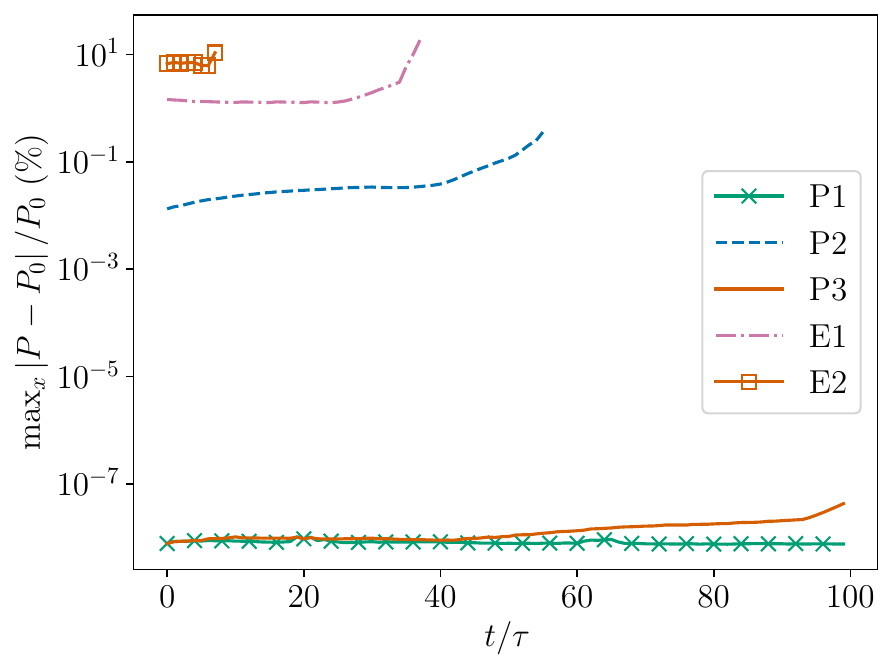}
\par\end{centering}
\caption{\label{fig:thermal_bubble_dodecane_deltaP_600_m-s_2D}Temporal variation
of error in pressure for $p=3$ solutions to the two-dimensional advection
of a nitrogen/n-dodecane thermal bubble. P1: uncorrected pressure--based
DG scheme; P2: pressure-based DG scheme with original correction term~(\ref{eq:correction-term-original});
P3: proposed pressure-based DG scheme (Section~\ref{subsec:correction-term-modified});
E1: total-energy-based DG scheme with overintegration; E2: total-energy-based
DG scheme with colocated integration; Exact: exact solution.}
\end{figure}

Figures~\ref{fig:thermal_bubble_2d_uncorrected} and~\ref{fig:thermal_bubble_2d_modified}
display the final density and temperature fields for the P1 and P3
schemes, respectively. The shape of the bubble is well-maintained
in both cases due to the absence of any pressure and velocity disturbances.

\begin{figure}[h]
\begin{centering}
\subfloat[\label{fig:thermal_bubble_2d_uncorrected_density}Density.]{\includegraphics[width=0.48\columnwidth]{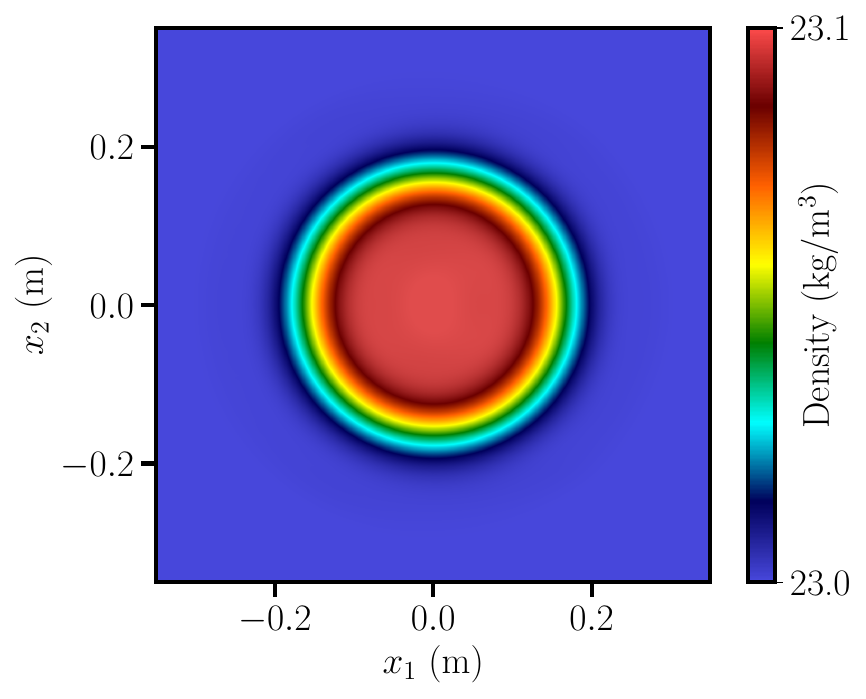}}\hfill{}\subfloat[\label{fig:thermal_bubble_2d_uncorrected_temperature}Temperature.]{\includegraphics[width=0.48\columnwidth]{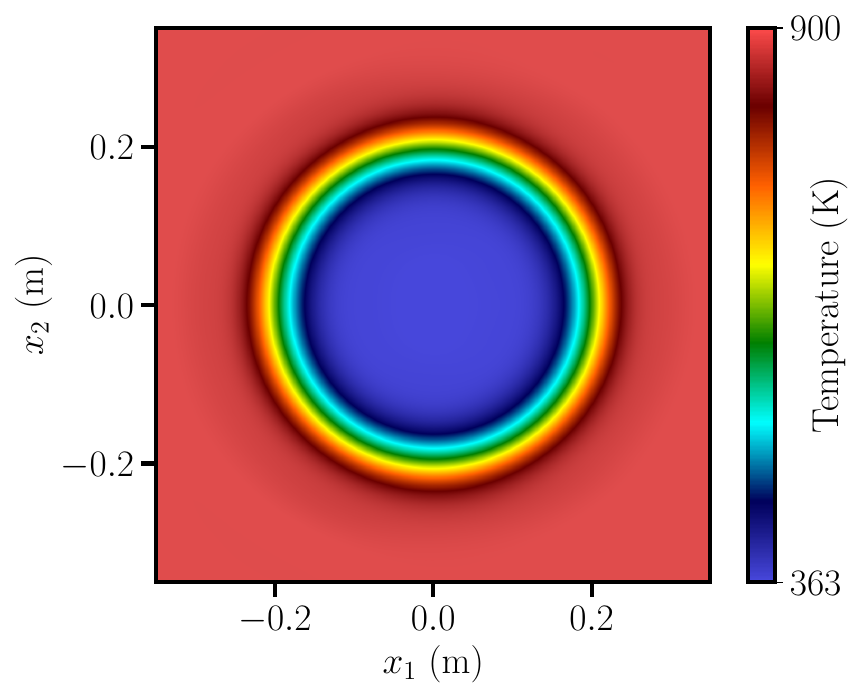}}
\par\end{centering}
\caption{\label{fig:thermal_bubble_2d_uncorrected}$p=3$ solution to two-dimensional
advection of a nitrogen/n-dodecane thermal bubble at $t=100\tau$
computed with the uncorrected DG scheme (P1). }
\end{figure}
\begin{figure}[h]
\begin{centering}
\subfloat[\label{fig:thermal_bubble_2d_modified_density}Density.]{\includegraphics[width=0.48\columnwidth]{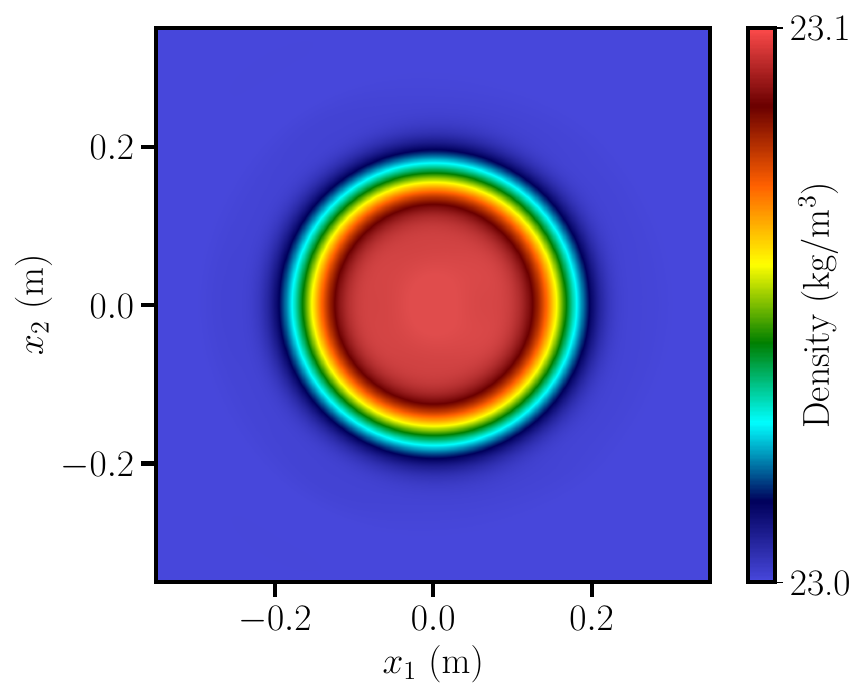}}\hfill{}\subfloat[\label{fig:thermal_bubble_2d_modified_temperature}Temperature.]{\includegraphics[width=0.48\columnwidth]{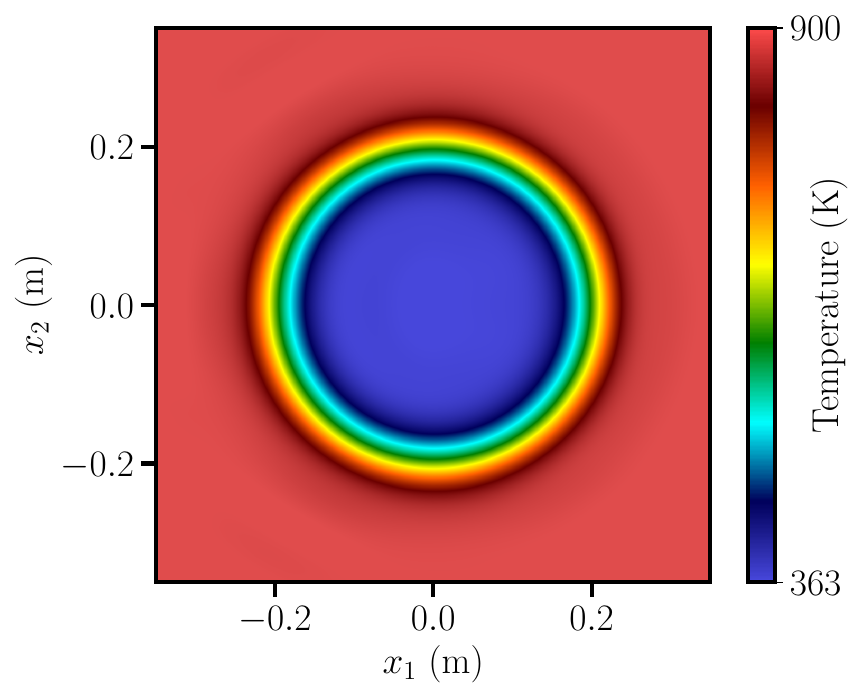}}
\par\end{centering}
\caption{\label{fig:thermal_bubble_2d_modified}$p=3$ solution to two-dimensional
advection of a nitrogen/n-dodecane thermal bubble at $t=100\tau$
computed with the proposed pressure-based DG scheme (P3). }
\end{figure}

\paragraph{Curved elements}

Finally, we recompute this case using curved elements of quadratic
order with a CFL of 0.4. To generate the curved mesh, high-order geometric
nodes are inserted at the midpoints of the vertices of each element.
At interior edges, the midpoint nodes are randomly perturbed by a
distance up to $0.05h$. Figure~\ref{fig:thermal_bubble_dodecane_deltaP_600_m-s_2D-curved}
presents the temporal variation of percent error in pressure for all
considered schemes. Again, only the P1 and P3 schemes remain stable
for 100 advection periods. Interestingly, the E2 scheme maintains
stability for longer times, and the P2 (original correction formulation)
solution diverges earlier than the E1 solution.
\begin{figure}[tbph]
\begin{centering}
\includegraphics[width=0.48\columnwidth]{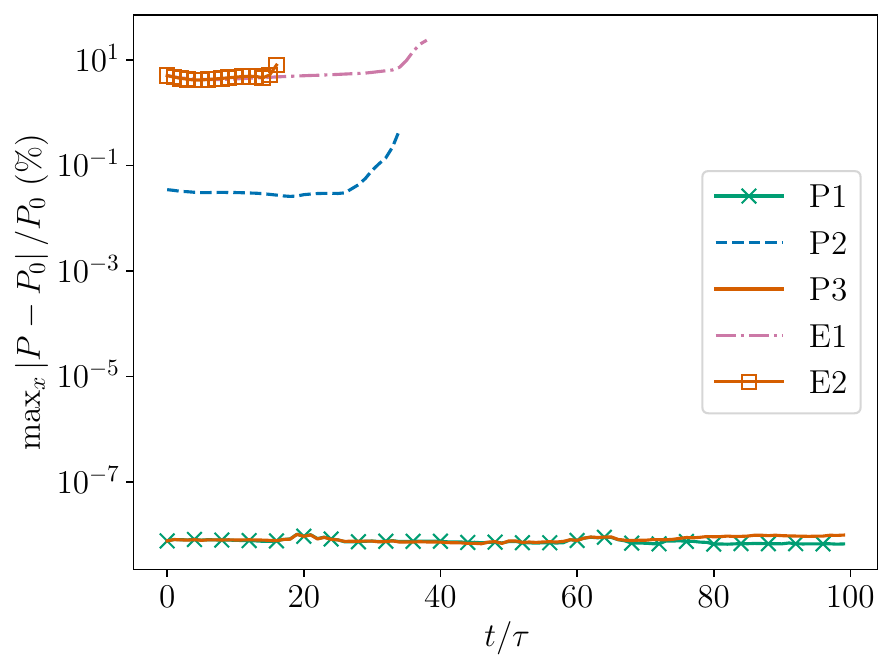}
\par\end{centering}
\caption{\label{fig:thermal_bubble_dodecane_deltaP_600_m-s_2D-curved}Temporal
variation of error in pressure for $p=3$ solutions to the two-dimensional
advection of a nitrogen/n-dodecane thermal bubble on a curved grid.
P1: uncorrected pressure--based DG scheme; P2: pressure-based DG
scheme with original correction term~(\ref{eq:correction-term-original});
P3: proposed pressure-based DG scheme (Section~\ref{subsec:correction-term-modified});
E1: total-energy-based DG scheme with overintegration; E2: total-energy-based
DG scheme with colocated integration; Exact: exact solution.}
\end{figure}
Figures~\ref{fig:thermal_bubble_2d_uncorrected-curved} and~\ref{fig:thermal_bubble_2d_modified-curved}
display the final density and temperature fields for the P1 and P3
schemes, respectively, superimposed by the curved grid. Just as with
the straight-sided grid, the shape of the bubble is well-maintained
in both cases due to the absence of any pressure and velocity disturbances. 

\begin{figure}[h]
\begin{centering}
\subfloat[\label{fig:thermal_bubble_2d_uncorrected_density-curved}Density.]{\includegraphics[width=0.48\columnwidth]{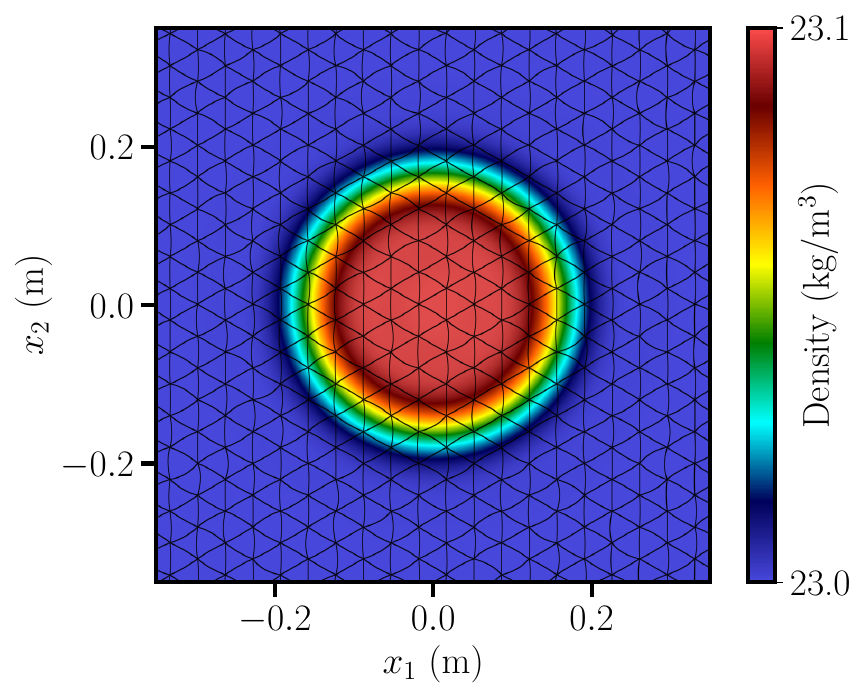}}\hfill{}\subfloat[\label{fig:thermal_bubble_2d_uncorrected_temperature-curved}Temperature.]{\includegraphics[width=0.48\columnwidth]{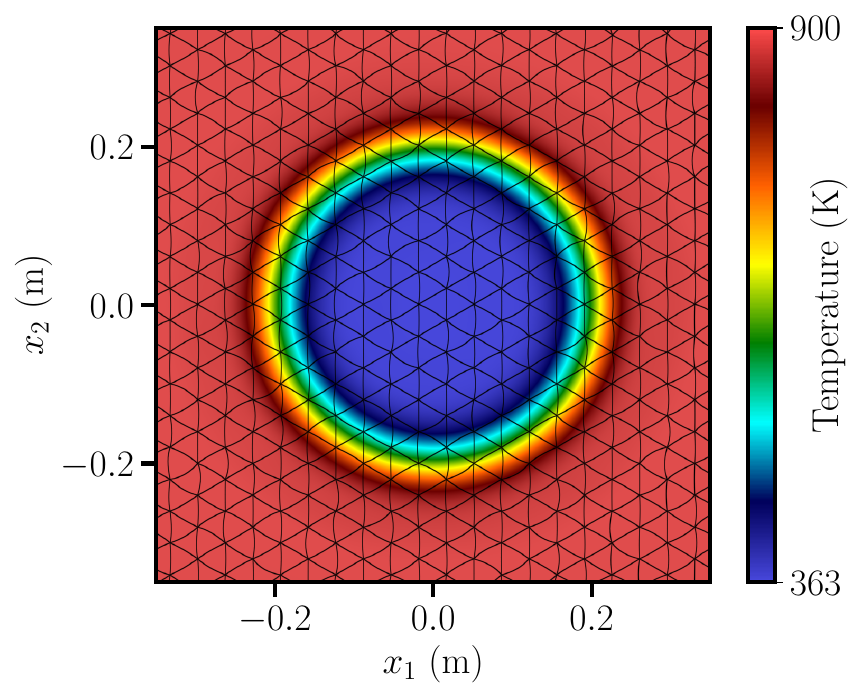}}
\par\end{centering}
\caption{\label{fig:thermal_bubble_2d_uncorrected-curved}$p=3$ solution to
two-dimensional advection of a nitrogen/n-dodecane thermal bubble
at $t=100\tau$ computed with the uncorrected DG scheme (P1) on a
curved grid.}
\end{figure}
\begin{figure}[h]
\begin{centering}
\subfloat[\label{fig:thermal_bubble_2d_modified_density-curved}Density.]{\includegraphics[width=0.48\columnwidth]{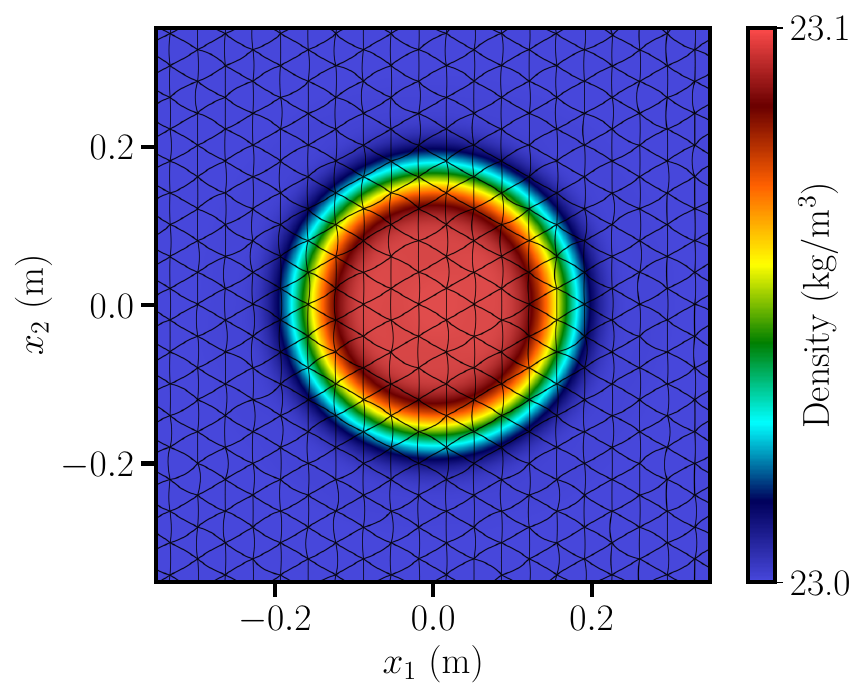}}\hfill{}\subfloat[\label{fig:thermal_bubble_2d_modified_temperature-curved}Temperature.]{\includegraphics[width=0.48\columnwidth]{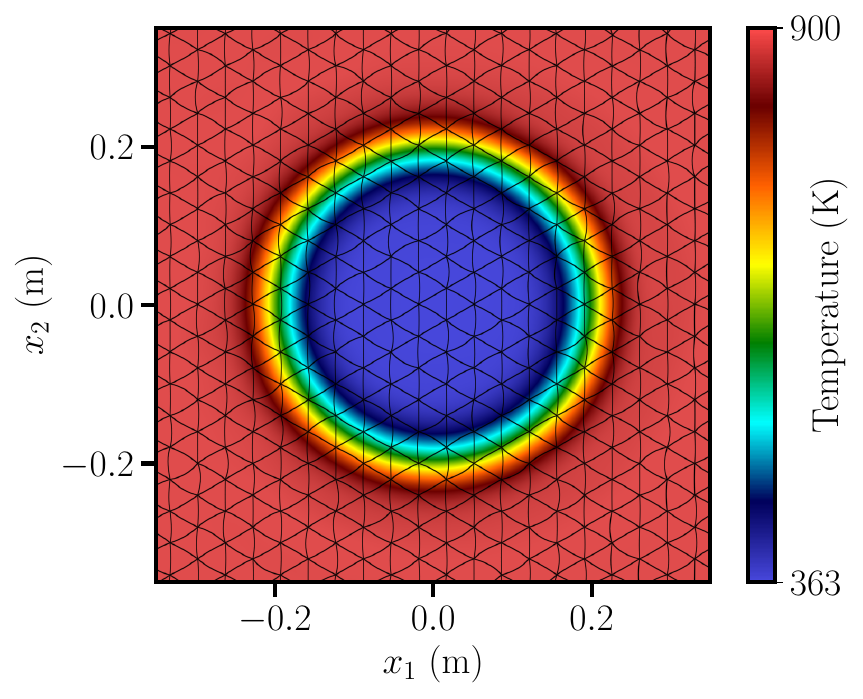}}
\par\end{centering}
\caption{\label{fig:thermal_bubble_2d_modified-curved}$p=3$ solution to two-dimensional
advection of a nitrogen/n-dodecane thermal bubble at $t=100\tau$
computed with the proposed pressure-based DG scheme (P3) on a curved
grid.}
\end{figure}

\subsection{Three-dimensional, high-velocity thermal bubble}

\label{subsec:thermal-bubble-600-m-s-3D}

Our final test case is a three-dimensional version of the high-velocity
thermal-bubble problem in Section~\ref{subsec:thermal-bubble-600-m-s-1D}.
The initial condition is given by
\begin{eqnarray}
Y_{n\text{-}\mathrm{C_{12}H_{26}}} & = & \frac{1}{2}\left[1-\tanh\left(25\sqrt{x_{1}^{2}+x_{2}^{2}+x_{3}^{2}}-5\right)\right],\nonumber \\
Y_{\mathrm{N_{2}}} & = & 1-Y_{n\text{-}\mathrm{C_{12}H_{26}}},\nonumber \\
T & = & \frac{T_{\min}+T_{\max}}{2}+\frac{T_{\max}-T_{\min}}{2}\tanh\left(25\sqrt{x_{1}^{2}+x_{2}^{2}+x_{3}^{2}}-5\right)\textrm{ K},\label{eq:thermal-bubble-2d-1}\\
P & = & 6\textrm{ MPa},\nonumber \\
\left(v_{1},v_{2},v_{3}\right) & = & \left(600,0,0\right)\textrm{ m/s},\nonumber 
\end{eqnarray}
The computational domain is $\text{\ensuremath{\Omega}}=\left[-0.5,0.5\right]\mathrm{m}\times\left[-0.5,0.5\right]\mathrm{m}\times\left[-0.5,0.5\right]\mathrm{m}$.
All boundaries are periodic. Gmsh~\citep{Geu09} is used to generate
an unstructured triangular mesh with a characteristic cell size of
$h=0.02~\mathrm{m}$. Figure~\ref{fig:thermal_bubble_3d_initial}
presents $T=890\;\mathrm{K}$ isosurfaces colored by density at $t=0$.

\begin{figure}[tbph]
\begin{centering}
\includegraphics[width=0.48\columnwidth]{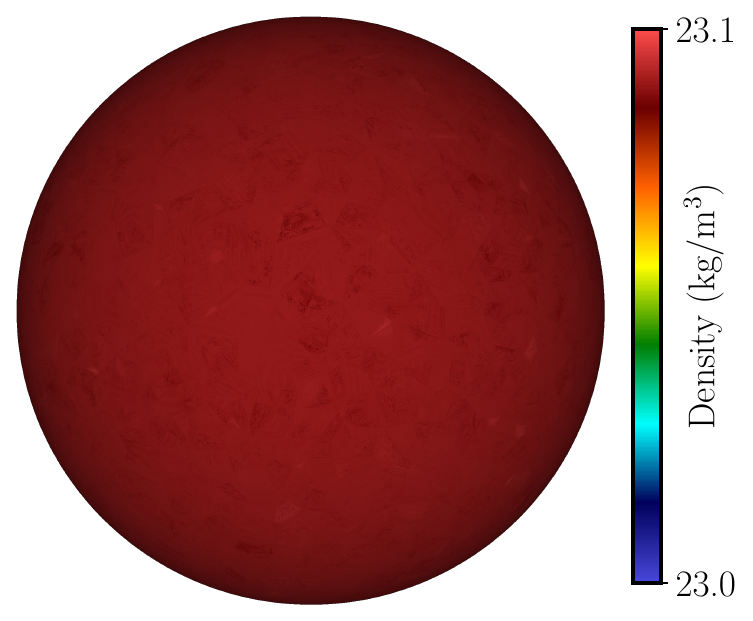}
\par\end{centering}
\caption{\label{fig:thermal_bubble_3d_initial}$T=890\;\mathrm{K}$ isosurfaces
colored by density at $t=0$ for the three-dimensional advection of
a nitrogen/n-dodecane thermal bubble.}
\end{figure}

Figure~\ref{fig:thermal_bubble_dodecane_deltaP_600_m-s_3D} presents
the temporal variation of percent error in pressure, sampled every
$\tau$ seconds, for all considered schemes. $p=3$ solutions are
integrated in time with a CFL of 0.6 until either $t=100\tau$ or
the solution diverges. As in the two-dimensional setting, only the
pressure-based DG schemes without any correction (P1) and with the
proposed correction formulation (P3) remain stable through $t=100\tau$.
The total-energy-based scheme with colocated integration (E2) diverges
the earliest, followed by the pressure-based DG scheme with the original
correction term (P2). The small pressure deviations observed for the
P1 and P3 schemes are due to finite-precision issues.
\begin{figure}[tbph]
\begin{centering}
\includegraphics[width=0.48\columnwidth]{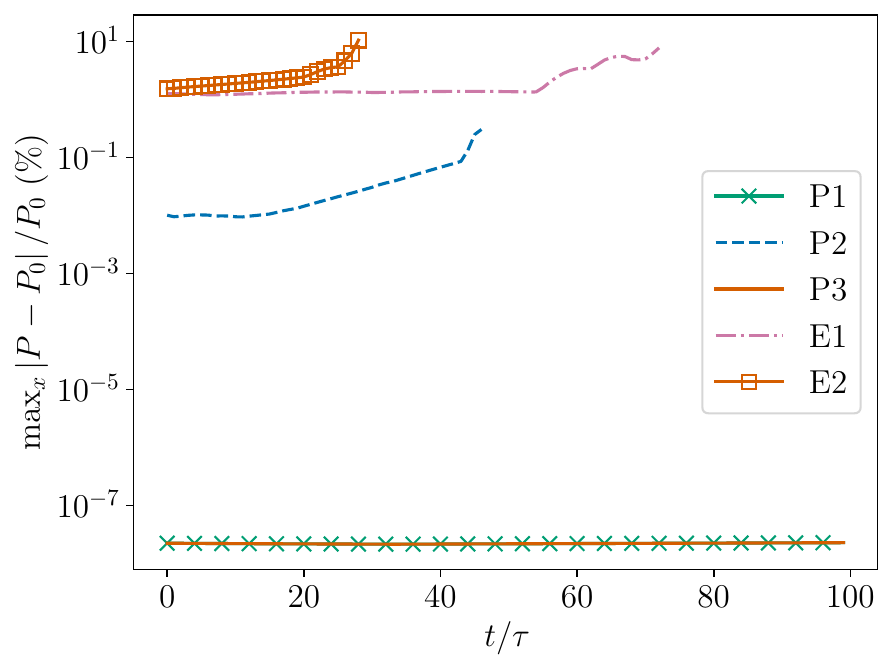}
\par\end{centering}
\caption{\label{fig:thermal_bubble_dodecane_deltaP_600_m-s_3D}Temporal variation
of error in pressure for $p=3$ solutions to the three-dimensional
advection of a nitrogen/n-dodecane thermal bubble. P1: uncorrected
pressure--based DG scheme; P2: pressure-based DG scheme with original
correction term~(\ref{eq:correction-term-original}); P3: proposed
pressure-based DG scheme (Section~\ref{subsec:correction-term-modified});
E1: total-energy-based DG scheme with overintegration; E2: total-energy-based
DG scheme with colocated integration; Exact: exact solution.}
\end{figure}

Figure~\ref{fig:thermal_bubble_3d_uncorrected_modified} the temperature
isosurfaces (colored by density) at $t=100\tau$ computed with the
P1 and P3 schemes, respectively. Although density variations are observed,
they remain marginal, and the shape of the bubble is well-maintained
in both cases.

\begin{figure}[h]
\begin{centering}
\subfloat[\label{fig:thermal_bubble_3d_uncorrected}Uncorrected scheme (P1).]{\includegraphics[width=0.48\columnwidth]{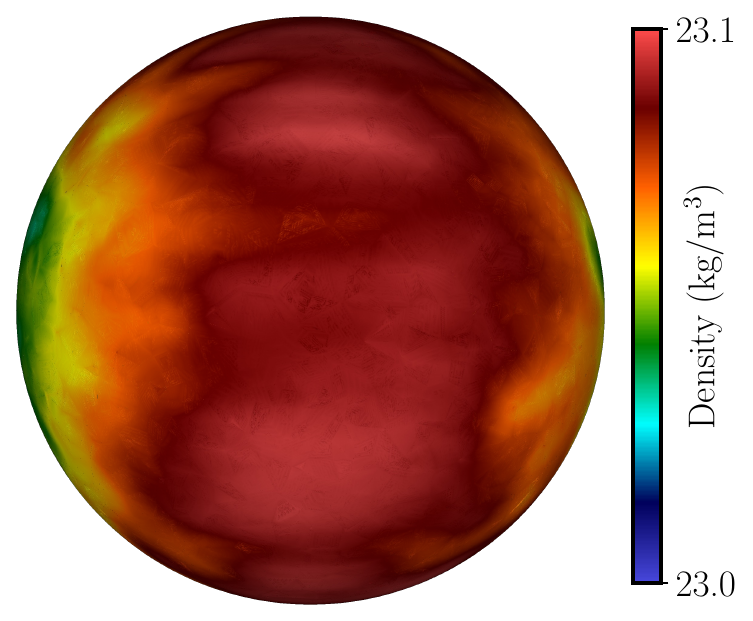}}\hfill{}\subfloat[\label{fig:thermal_bubble_3d_modified}Proposed scheme (P3).]{\includegraphics[width=0.48\columnwidth]{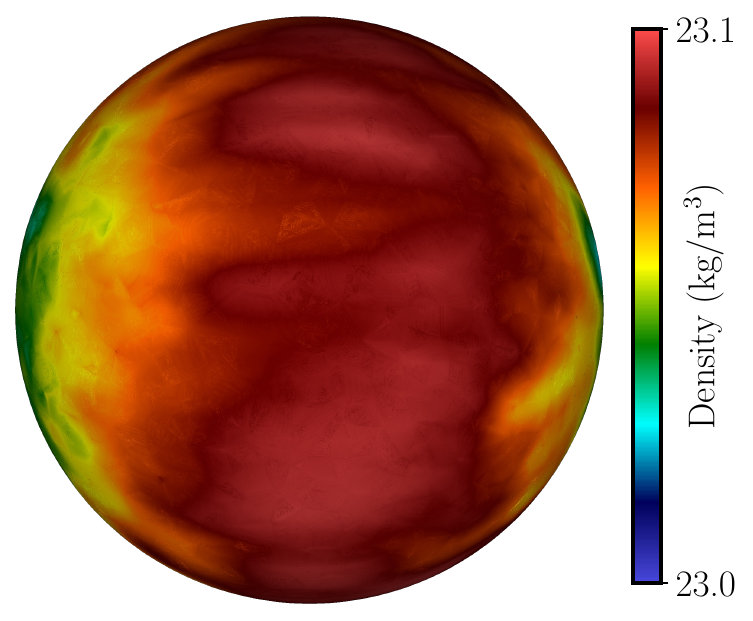}}
\par\end{centering}
\caption{\label{fig:thermal_bubble_3d_uncorrected_modified}$T=890\;\mathrm{K}$
isosurfaces colored by density at $t=100\tau$ for the three-dimensional
advection of a nitrogen/n-dodecane thermal bubble computed with the
P1 and P3 schemes.}
\end{figure}

\section{Concluding remarks}

In this work, we analyzed the velocity-equilibrium and pressure-equilibrium
conditions of a standard DG scheme that discretizes the conservative
form of the compressible, multicomponent Euler quations. It was shown
that under certain constraints on the numerical flux, the scheme is
velocity-equilibrium-preserving. However, standard DG schemes are
not pressure-equilibrium-preserving. Therefore, a pressure-based formulation
was adopted (i.e., the total-energy conservation equation is replaced
by pressure-evolution equation). In order to restore the (semi-discrete)
total-energy conservation that is otherwise lost, we incorporated
the conservative, elementwise correction terms of Abgrall~\citep{Abg18}
and Abgrall et al.~\citep{Abg22}. Unfortunately, the corrected pressure-based
DG scheme no longer exactly preserves pressure and velocity equilibria.
Furthermore, it does not preserve zero species concentrations. To
address these issues, we proposed simple modifications to the correction
term that then enable preservation of pressure equilibrium, velocity
equilibrium, and zero species concentrations, in addition to semidiscrete
conservation of total energy. Since the elementwise correction terms
are not valid if the state inside a given element is uniform, we also
introduced face-based corrections, inspired by~\citep{Abg23}, along
with a particular definition of the numerical total-energy flux, to
account for the case of elementwise-constant solutions with inter-element
jumps. We applied the scheme to compute smooth, interfacial flows
initially in velocity and pressure equilibria. In the first test case,
optimal convergence was demonstrated. However, the proposed modifications
may introduce additional errors that primarily manifest in $p=1$
solutions (although $p\geq2$ solutions seem largely unaffected).
The next two test cases entailed high-velocity and low-velocity advection
of a nitrogen/n-dodecane thermal bubble in one dimension. A total-energy-based
DG formulation with colocated integration failed to maintain solution
stability in the former, while a total-energy-based formulation with
overintegration failed to maintain stability in the latter. The pressure
oscillations were significantly reduced, but not eliminated, using
the pressure-based DG scheme with the original correction term. With
the proposed modifications, both pressure-equilibrium preservation
and semidiscrete conservation of total energy were achieved. In two-
and three-dimensional versions of the high-velocity nitrogen/n-dodecane
thermal-bubble problem, the pressure-based DG scheme with the original
correction term resulted in solver divergence, while the developed
scheme maintained solution stability due to exact preservation of
pressure equilibrium (apart from minor finite-precision-induced errors).

Note that apart from pressure-equilibrium preservation, there are
certain advantages to using a pressure-based formulation instead of
a total-energy-based formulation. For example, in the former, temperature
is straightforward to compute from the state; in contrast, an iterative
solver is typically required to obtain temperature from the state
in the latter. In addition, the occurrence of negative pressures is
a major issue in underresolved smooth regions and near discontinuities
that can lead to solver divergence; it is possible that negative pressures
are more likely to happen in a total-energy-based formulation since
pressure is a nonlinear function of the state, although this requires
further investigation. In certain cases, it may be easier to guarantee
positive pressures in a pressure-based formulation since pressure
is a state variable (though this is outside the scope of the current
study). At the same time, there are disadvantages to using a pressure-based
formulation as well. For instance, in the case of viscous flows, the
discretization of the additional terms may not be straightforward.
Furthermore, limiters that are conservative in total-energy-based
formulations do not necessarily conserve total energy in pressure-based
formulations. Entropy-based stabilization techniques that rely on
a total-energy-based formulation may also be difficult to apply.

There are many potential directions for future work. The current scheme
is designed for smooth flows initially in velocity and pressure equilibria;
we plan to extend it to flows with discontinuities by incorporating
appropriate stabilization mechanisms. The modified correction term
may not be appropriate in the case of non-uniform velocity and pressure,
in which case the original correction term is more suitable; therefore,
we may adopt a hybrid strategy that switches between the two correction
terms in such a way that guarantees semidiscrete conservation of total
energy and, where appropriate, preservation of pressure and velocity
equilibrium. Another hybrid approach is to switch to a total-energy
formulation when pressure-equilibrium preservation is not applicable,
although this may introduce new complications; note that such an approach
would be distinct from other hybrid schemes in the literature that
combine total-energy-based and pressure-based formulations due to
semidiscrete satisfaction of total-energy conservation in the entire
domain. To handle discontinuities, we will explore stabilization techniques
that do not cause loss of these properties. Furthermore, we will
conduct a detailed investigation of the importance of conserving total
energy in regions away from shocks, particularly in large-scale, practical
flow problems. An additional matter we aim to address is the importance
of fully discrete (as opposed to semidiscrete) conservation of total
energy, which would require the use of specific time integrators.
Finally, we will consider viscous effects, chemical reactions, and
real-fluid mixtures, which are much more susceptible to spurious pressure
oscillations and other nonlinear instabilities, and perform in-depth
comparisons with state-of-the-art total-energy-based formulations.

\section*{Acknowledgments}

This work is sponsored by the Office of Naval Research through the
Naval Research Laboratory 6.1 Computational Physics Task Area. 

\bibliographystyle{elsarticle-num}
\bibliography{../../../JCP_submission/citations}

\appendix

\section{Compressible vortex transport: Grid convergence}

\label{sec:appendix-compressible-vortex-transport}

In this section, we study convergence under grid refinement of the
DG discretization~(\ref{eq:DG-discretization-nonconservative}) in
order to test our implementation of the nonconservative terms, especially
since the cases in Section~\ref{sec:results} involve flows initially
in pressure and velocity equilibria, such that the nonconservative
terms vanish (assuming pressure equilibrium is discretely preserved).
We consider vortex advection based on the configuration presented
in~\citep{Lv20}. The initial condition is written in nondimensional
form as
\begin{eqnarray*}
v_{1} & = & v_{\infty}-\frac{\sigma}{2\pi}\left(x_{2}-x_{c,2}\right)\exp\left(1-r^{2}\right),\\
v_{2} & = & v_{\infty}+\frac{\sigma}{2\pi}\left(x_{1}-x_{c,1}\right)\exp\left(1-r^{2}\right),\\
\rho & = & \left[1-\frac{\sigma^{2}\left(\gamma-1\right)}{16\pi^{2}\gamma}\exp\left(2-2r^{2}\right)\right]^{\frac{1}{\gamma-1}},\\
P & = & \rho^{\gamma},
\end{eqnarray*}
where $v_{\infty}=10$, $r=\sqrt{\left(x_{1}-x_{c,1}\right)^{2}+\left(x_{2}-x_{c,2}\right)^{2}}$,
$\sigma=4$ is the vortex strength, and $\gamma=1.4$. The computational
domain is $\text{\ensuremath{\Omega}}=\left[0,L\right]^{2}$, where
$L=10$, and the vortex center is $\left(x_{c,1},x_{c,2}\right)=\left(L/2,L/2\right).$
All boundaries are periodic, and all simulations are run for one advection
period with a CFL of 0.1, as defined in Equation~(\ref{fig:thermal_bubble_3d_initial}).
Figure~\ref{fig:vortex_convergence} presents the $L^{2}$ errors
of the state variables (normalized as discussed in Section~\ref{subsec:gaussian-density-wave})
for $p=1$ to $p=3$ and four grid sizes, where the coarsest grid
size is $h=1$. The dashed lines represent convergence rates of $p+1$.
Optimal convergence is observed.
\begin{figure}[tbph]
\begin{centering}
\includegraphics[width=0.48\columnwidth]{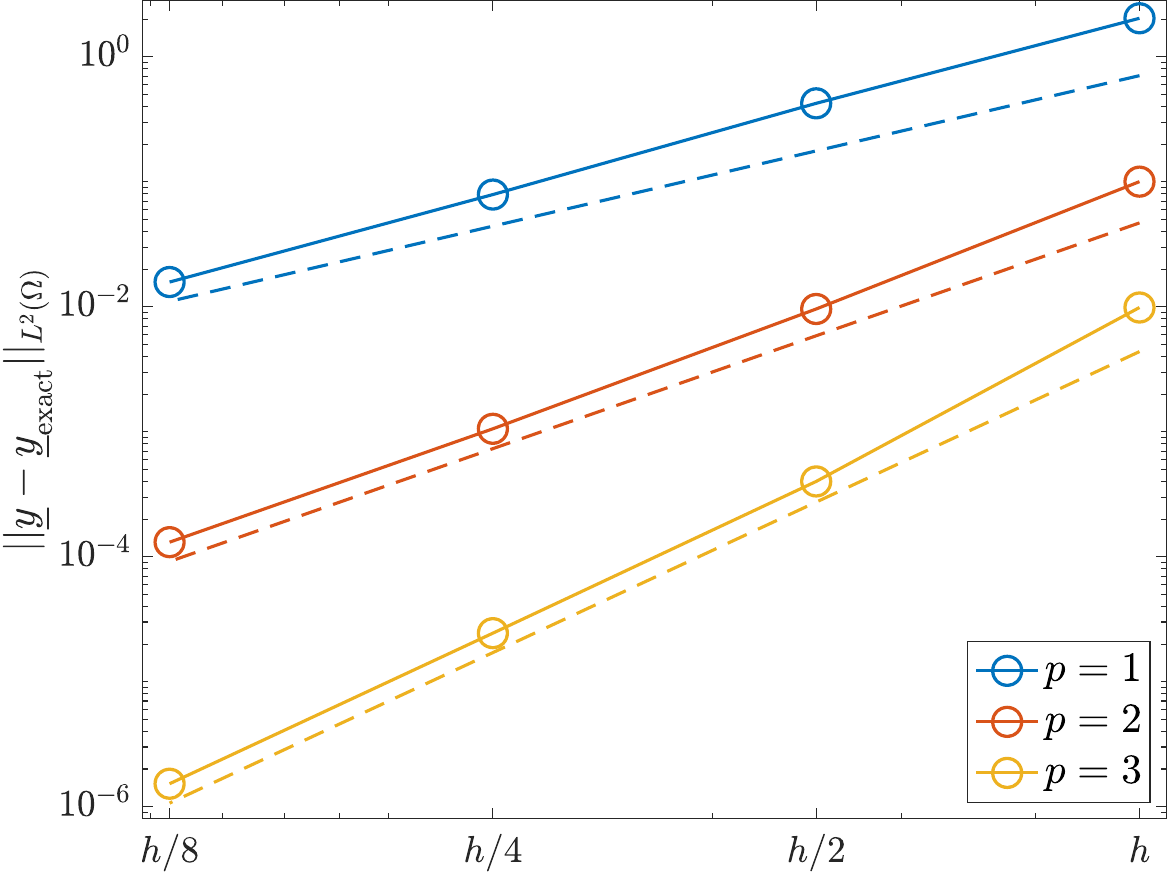}
\par\end{centering}
\caption{\label{fig:vortex_convergence}Convergence under grid refinement,
with $h=1$, for the vortex-transport problem. The $L^{2}$ error
of the normalized state with respect to the exact solution after one
advection period is computed. The dashed lines represent convergence
rates of $p+1$.}
\end{figure}

\section{Derivative of total energy with respect to the state}

\label{sec:appendix-derivative-total-energy}

This section provides a derivation of $\bm{w}=\partial_{\bm{y}}\left(\rho e_{t}\right)$
in Equation~(\ref{eq:total-energy-derivative}). The perturbation
of total energy, $\rho e_{t}=\rho u+\frac{1}{2}\rho\sum_{k=1}^{d}v_{k}v_{k}$,
is given by
\begin{equation}
\delta\rho e_{t}=\left(\frac{\partial\rho u}{\partial T}\right)_{C_{i}}\delta T+\sum_{i=1}^{n_{s}}\left(\frac{\partial\rho u}{\partial C_{i}}\right)_{T,C_{j\neq i}}\delta C_{i}+\sum_{k=1}^{d}v_{k}\delta\left(\rho v_{k}\right)-\frac{1}{2}\sum_{k=1}^{d}v_{k}v_{k}\delta\rho.\label{eq:total-energy-perturbation}
\end{equation}
With $T=T\left(C_{1},\ldots,C_{n_{s}},P\right)$, the temperature
perturbation can be written as
\begin{equation}
\delta T=\left(\frac{\partial T}{\partial P}\right)_{C_{i}}\delta P+\sum_{i=1}^{n_{s}}\left(\frac{\partial T}{\partial C_{i}}\right)_{P,C_{i\neq i}}\delta C_{i}\label{eq:temperature-perturbation}
\end{equation}
Substituting Equation~(\ref{eq:temperature-perturbation}) into Equation~(\ref{eq:total-energy-perturbation})
yields
\begin{align*}
\delta\rho e_{t}= & \left(\frac{\partial\rho u}{\partial T}\right)_{C_{i}}\left(\frac{\partial T}{\partial P}\right)_{C_{i}}\delta P+\left(\frac{\partial\rho u}{\partial T}\right)_{C_{i}}\sum_{i=1}^{n_{s}}\left(\frac{\partial T}{\partial C_{i}}\right)_{P,C_{i\neq i}}\delta C_{i}+\sum_{i=1}^{n_{s}}\left(\frac{\partial\rho u}{\partial C_{i}}\right)_{T,C_{j\neq i}}\delta C_{i}\\
 & +\sum_{k=1}^{d}v_{k}\delta\left(\rho v_{k}\right)-\frac{1}{2}\sum_{k=1}^{d}v_{k}v_{k}\sum_{i=1}^{n_{s}}W_{i}\delta C_{i}\\
= & \sum_{k=1}^{d}v_{k}\delta\left(\rho v_{k}\right)+\left(\frac{\partial\rho u}{\partial T}\right)_{C_{i}}\left(\frac{\partial T}{\partial P}\right)_{C_{i}}\delta P+\sum_{i=1}^{n_{s}}\left[\left(\frac{\partial\rho u}{\partial T}\right)_{C_{i}}\left(\frac{\partial T}{\partial C_{i}}\right)_{P,C_{i\neq i}}+\left(\frac{\partial\rho u}{\partial C_{i}}\right)_{T,C_{j\neq i}}-\frac{W_{i}}{2}\sum_{k=1}^{d}v_{k}v_{k}\right]\delta C_{i},
\end{align*}
which gives

\begin{align*}
\bm{w} & =\frac{\partial\left(\rho e_{t}\right)}{\partial\bm{y}}\\
 & =\left(\begin{array}{ccc}
\frac{\partial\left(\rho e_{t}\right)}{\partial\rho v_{k}}, & \frac{\partial\left(\rho e_{t}\right)}{\partial P}, & \frac{\partial\left(\rho e_{t}\right)}{\partial C_{i}}\end{array}\right)^{T}\\
 & =\left(\begin{array}{ccc}
v_{k}, & \left(\frac{\partial\rho u}{\partial T}\right)_{C_{i}}\left(\frac{\partial T}{\partial P}\right)_{C_{i}}, & \left(\frac{\partial\rho u}{\partial T}\right)_{C_{i}}\left(\frac{\partial T}{\partial C_{i}}\right)_{P,C_{i\neq i}}+\left(\frac{\partial\rho u}{\partial C_{i}}\right)_{T,C_{j\neq i}}-\frac{W_{i}}{2}\sum_{k=1}^{d}v_{k}v_{k}\end{array}\right)^{T}.
\end{align*}
The equation of state~(\ref{eq:eos}) results in
\[
\left(\frac{\partial T}{\partial P}\right)_{C_{i}}=\frac{1}{R^{0}\sum_{i}C_{i}},\quad\left(\frac{\partial T}{\partial C_{i}}\right)_{P,C_{j\neq i}}=-\frac{P}{R^{0}\left(\sum_{i}C_{i}\right)^{2}}.
\]
Introducing the mass-specific heat capacity at constant volume of
the $i$th species, $c_{v,i}=c_{p,i}-R_{i}$, and $\rho c_{v}=\sum_{i=1}^{n_{s}}\rho_{i}c_{v,i}$,
we have
\[
\left(\frac{\partial\rho u}{\partial T}\right)_{C_{i}}=\rho c_{v},\quad\left(\frac{\partial\rho u}{\partial C_{i}}\right)_{T,C_{j\neq i}}=W_{i}u_{i}.
\]
$\bm{w}$ can then be expressed as
\[
\bm{w}=\left(\begin{array}{ccc}
v_{k}, & \frac{\rho c_{v}}{R^{0}\sum_{j}C_{j}}, & W_{i}u_{i}-\frac{\rho c_{v}P}{R^{0}\left(\sum_{j}C_{j}\right)^{2}}-\frac{W_{i}}{2}\bm{v}\cdot\bm{v}\end{array}\right)^{T}.
\]

\end{document}